\documentclass[11pt,a4paper]{amsart}

\usepackage{caption}
\usepackage{subcaption}

\usepackage{fancyhdr}
\usepackage{setspace}
\usepackage{amsmath,amsthm}
\usepackage{amssymb, amsfonts}
\usepackage[dvips]{graphicx}

\usepackage[ansinew]{inputenc}
\usepackage{enumerate}
\usepackage{bbm}

\usepackage{verbatim}

\usepackage{color}
\usepackage{transparent}

\newtheorem{theorem}{Theorem}[section]
\newtheorem{corollary}[theorem]{Corollary}
\newtheorem{lemma}[theorem]{Lemma}
\newtheorem{proposition}[theorem]{Proposition}

\theoremstyle{definition}

\newtheorem{definition}[theorem]{Definition}
\newtheorem{assumption}[theorem]{Assumption}
\newtheorem{agreement}[theorem]{Agreement}
\newtheorem{descrip}[theorem]{Description}

\theoremstyle{remark}
\newtheorem{remark}[theorem]{Remark}

\newtheorem{example}[theorem]{Example}

\newcommand{\PP}{{\mathbb P}}
\newcommand{\EE}{{\mathbb E}}

\newcommand{\R}{\mathbb{R}}
\newcommand{\N}{\mathbb{N}}

\newcommand{\argmax}{\operatornamewithlimits{argmax}}

\newcommand{\E}{\mathbb{E}}
\newcommand{\p}{\mathbb{P}}
\newcommand{\I}{\mathcal{I}}
\newcommand{\Var}{\text{Var}}

\begin{document}
\title{Optimal variance stopping with linear diffusions}
\author{Kamille Sofie T{\aa}gholt Gad}
\address{University of Copenhagen\\
Universitetsparken 5\\
2100 Copenhagen\\
Denmark}
\email{kamille@math.ku.dk}

\author{Pekka Matom{\"a}ki}
\address{Turku School of Economics\\
Department of Accounting and Finance\\
20014 University of Turku\\
Finland}
\email{pmatomaki@gmail.com}


Published in Stochastic Processes and their Applications, Vol 130(4), 2349 -- 2383 (2020). 

\begin{abstract}
We study the optimal stopping problem of maximizing the variance of an unkilled linear diffusion. Especially, we demonstrate how the problem can be solved as a convex two-player zero-sum game, and reveal quite surprising application of game theory by doing so. Our main result shows that an optimal solution can, in general case, be found among stopping times that are mixtures of two hitting times. This and other revealed phenomena together with suggested solution methods could be helpful when facing more complex non-linear optimal stopping problems. The results are illustrated by a few examples.
\end{abstract}
\maketitle
\section{Introduction}

In classical optimal stopping problems one seeks a stopping time that optimizes the expectation of some process upon stopping, and optimal stopping with respect to higher moments has only recently been approached (see \cite{Gad15,GadPed15,Pedersen11,PedPes12, Christensen18}). 
In this paper we study the optimal stopping problem of finding a stopping time that maximizes the variance of a general unkilled linear diffusion $X_t$, i.e. we study the \emph{variance problem} 
\begin{align}\label{eq prob}
\sup_{\tau\in \mathcal{T}}\Var_x\left\{X_\tau\right\}=\sup_{\tau\in \mathcal{T}}\left\{\E_x\left\{X_\tau^2\right\}-\E_x\left\{X_\tau\right\}^2\right\},
\end{align}
where $\mathcal{T}$ is a class of randomized stopping times generated from the filtration of $X$. Loosely, a randomized stopping times is a stopping time where one chooses a stopping time using a known probability function. The main difficulty in \eqref{eq prob} is the fact that due to the non-linear term on the right hand side, the highly developed machinery for solving classical optimal stopping problems (e.g. \cite{BeiLer00,Pedersen05,Salminen85,PesShi06}) is not readily usable. 

The field of non-linear optimal stopping is new and thus present work deals with basic questions. The main importance on the results is thus the structure of solutions and the identification of tools for solution methods, rather than the specific solutions. However, the variance stopping problem is also interesting in its own right: Variance may be seen as a measure of risk and by maximizing the variance we get a tight upper bound for this risk.

As our main results, we have two observations to offer. The first observation is that for a general non-killed continuous diffusion the pure threshold rule cannot provide the value for all cases; There are cases where a randomized mixture of two exit times is an optimal stopping time and simple threshold rules offer purely weaker values. A somewhat similar result has been shown for some processes with jumps (see \cite{GadPed15}). However, the importance of randomized stopping times for the jump processes studied in \cite{GadPed15} can be narrowed down to mixtures of exit times with a single boundary. For general diffusions, we need to randomized between one and two boundary exit times to reach the solution. This reliance on randomized stopping times is a remarkable difference from classical optimal stopping problems, which usually always have an exit time solution whenever a solution exists. 
The second observation we make is that the variance stopping problem is very closely related to game theory, and thus well-known results from game theory are readily usable. To see the link, we notice that for a random variable $X_t$, we have $\sup_{\tau}\Var_x(X_\tau)=\sup_{\tau}\inf_c\E_x\left\{(X_\tau-c)^2\right\}$. By first narrowing the class of stopping times within which the optimal stopping time is found, one can utilize the theory of continuous convex two-player zero-sum games (see e.g. \cite{Bohnenblust50,Karlin59, Vor77}) to solve the problem. 

Although usual optimal stopping problems do not give rise to randomized solution, in game theory context one can nevertheless find randomized optimal stopping times more easily. For example, in \cite{HHOZ15} it was shown how in a non-linear gambling problem an optimal stopping strategy of a pre-committing gambler can be a randomized stopping time. Further, in a Dynkin game setting --- a two-player stopping game --- one can find in some cases an optimal randomized stopping time (see \cite{Touzi02}). However, it should be noticed that in variance problem the randomized stopping time concept arises because the usual set of stopping rules is simply too limited to supply an optimal strategy in as complex case as a non-linear stopping set is. Very recently, the idea to formulate non-linear optimal stopping problems into game framework has been also considered  deeply in \cite{Christensen17,Christensen18}, and this aspect is discussed more deeply in Subsection \ref{subsec on optimality}.

We also have some minor observations to offer. One observation is that similar to previously studied non-linear optimal stopping problems (cf. \cite{GadPed15,Pedersen11,PedPes12}) the variance problem for general non-killed continuous diffusions has a strong dependence on the initial value of the underlying process. This translates into static optimality, which will be discussed more deeply in Subsection \ref{subsec on optimality}. Another observation is that we encounter quite a strong transiency requirement in order to attain a non-trivial solution: if $X_t$ is recurrent (i.e. hits every point with probability 1) or it is "not transient enough" then the problem is trivial. Lastly, we will observe that if the scale function satisfies a simple, typically satisfied monotonicity requirement, then randomization is not needed, and solution turns out to be quite effortless to find. All in all, our results indicate that although the variance problem is only a small step away from the classical linear optimal stopping problems, there are some quite substantial dissimilarities in the structure of the solutions as well as solution methods.

The study on variance stopping began recently when Pedersen in \cite{Pedersen11} proposed a verification theorem for the variance stopping problem \eqref{eq prob} and used it for some explicit continuous It\^o-diffusions. The verification theorem states that in order to reach the solution, it is sufficient to solve an embedded classical optimal stopping problem with certain side conditions. The verification theorem has also been used successfully for solving the variance problem for geometric L\'evy processes in \cite{GadPed15}. Another non-linear optimal stopping problem which has been solved explicitly is the \textit{Mean-variance problem} given by
\begin{align*}
\sup_\tau\left(\E_x\left\{X_\tau\right\}-c\Var_x\left\{X_\tau\right\}\right).
\end{align*}
This problem has been resolved in \cite{PedPes12} when the underlying process $X$ is a geometric Brownian motion applying a Lagrange multipliers method. 
Similar approach has been utilized in the mean-variance setting in \cite{Gad15b} for certain geometric L\'evy processes. Recently, more comprehensive view over general non-linear optimal stopping problem has been investigated in \cite{Christensen17,Huang17, Huang18, Christensen18}. These will be discussed in more detail in Subsection \ref{subsec on optimality}, as well as the notions on static and dynamic optimality.

The structure of this paper is the following. In Section \ref{sec formula} we lay down our assumptions and divide the problem into different categories depending on whether the solution is trivial or not. In Section \ref{sec results} we give our main results for non-trivial cases and discuss the meaning of the results. The proof of these results are laid down in Sections \ref{sec proof} and \ref{sec cases II & III} by leaning on knowledge on continuous convex zero-sum games. After the general proof, in Subsection \ref{subsec on optimality}, we discuss at length on optimality and structure of the solution. Trivial and marginal cases are discussed in Section \ref{sec special case} as well as killed diffusion case. In Section \ref{sec examples} we provide a step-by-step solution algorithm and illustrate our results with three examples. 

\section{Setting and problem formulation}\label{sec formula}

The mathematical formulation of our setting and problem are next considered.
\begin{agreement}\label{def X}
Let $X_t$ be a regular linear diffusion defined on a filtered probability space $(\Omega,\mathcal{F},\mathbb{F},\mathbb{P})$, where $\mathbb{F}$ is the filtration generated from $X$. Let $X_t$ evolve on $\I:=(\alpha,\beta)\subseteq\R$. The boundaries can be natural, exit, entrance or killing and in the case of killing boundaries, the end points belong to $\I$. We also allow absorbing boundaries by assuming that on the interior of $\I$ the diffusion is otherwise regular (i.e. starting from interior of $\I$, there is a positive probability that the process hits an arbitrary interior state). We assume that the diffusion does not die inside the state space and that the scale function $S(x)$ and speed measure $m(x)$ are continuous. 
\end{agreement}

The minimal requirements for a scale function are that it is increasing and continuous. Furthermore, we assume, without losing the generality, that either $S(\alpha)=0$ or $S(\alpha)=-\infty$. If $S(\alpha)\in(-\infty,\infty)$, we can define $\tilde{S}(x)=S(x)-S(\alpha)$ to be a new scale function fulfilling our assumption. Here, and later, we understand $S(\alpha)=\lim_{a\searrow\alpha}S(a)$.

Now, in the variance problem we seek to identify the value function, $V$, and an optimal stopping time, $\tau^*$, such that  
\begin{align}\label{eq probb}
V(x):=\sup_{\tau\in\mathcal{T}}\Var_x\left\{X_\tau\right\}=\Var_x\left\{X_{\tau^*}\right\}, 
\end{align}
where the subscript $x$ refers to the initial state of the process $X$, and $\mathcal{T}$ is the class of randomized stopping times, defined below, generated from $X$. 

We include stopping times which may take the value infinity, and thus it is common to rather refer to $\mathcal{T}$ as a set of Markov times (e.g. \cite{PesShi06}).

Since we allow stopping times to take the value infinity we need to specify how we interpret $X_\infty$. Let $\zeta:=\inf\{t\geq0\mid X_t =\alpha\text{ or }\beta\}$ be the life time of the process $X_t$ on the interior of $\I$. We interpret $X_\infty:=\limsup_{t\rightarrow \zeta}X_t$.

\begin{definition}\label{def random stopping time}
Define the class of \emph{randomized stopping times} in the following way. Assume that there exists a random variable $U$ uniformly distributed on $[0,1]$ and independent of the process $X$. This may require to expand the probability space. Define an augmented filtration $\hat{\mathbb{F}}$ as the filtration generated from both $U$ and the process $X$, that is $\hat{\mathbb{F}}_t=\sigma\{U,(X_s)_{0\leq s\leq t}\}$. Now the class of randomized stopping times $\mathcal{T}$ is defined as all stopping times generated from $\hat{\mathbb{F}}$. 
\end{definition}

This definition of randomized stopping times is quite general, and in fact more general than needed. We will see that whenever an optimal stopping time exists, we may find an optimal stopping time within the subclass denoted by \emph{Bernoulli randomized stopping times}. 
\begin{definition}\label{def Bernoulli}
The \emph{Bernoulli randomized stopping times} are the stopping times $\tau$ which may be written in the form: $\tau= \mathbbm{1}_{(U <p)}\tau_1 + \mathbbm{1}_{(U\geq p)}\tau_2$, where $p\in [0,1]$ and where $\tau_1, \tau_2$ are stopping times with respect to the filtration $\mathbb{F}$ and $U$ uniformly distributed on $[0,1]$ and independent of the process $X$. Equivalently, we may also write $\tau=\xi_p\tau_1+(1-\xi_{p})\tau_2$, where $\xi_p$ is a Bernoulli random variable with a parameter $p\in[0,1]$, i.e. $\p\left(\xi_p=1\right)=1-\p\left(\xi_p=0\right)=p$. We say that $\tau$ is a mixture of $\tau_1$ and $\tau_2$.
\end{definition}

The solutions we find for the variance problem are hitting times or mixtures of hitting times. For hitting times we use the following notation. 
We denote by $\tau_z:=\inf\{t\geq0\mid X_t=z\}$ the first hitting time to a state $z$ and by $\tau_{(z,y)}:=\inf\{t\geq0 \mid X_t\notin (z,y)\}$ the first exit time from an open interval $(z,y)$. Lastly, $\tau_{(\alpha,z)}=\lim_{a\to\alpha}\tau_{(a,z)}$ and similarly $\tau_{(z,\beta)}=\lim_{b\to\beta}\tau_{(z,b)}$.

\subsection{Scale function and transiency of a diffusion}

The next definition is essential when determining finiteness of the value and the shape of the solution.

\begin{definition}
\begin{enumerate}[(A)]
	\item A boundary point $\alpha$ ($\beta$) is \emph{attractive}, if $\lim_{t\to\infty}X_t=\alpha$ $(\beta)$ with positive probability. 
	\item A diffusion is said to be \emph{recurrent}, if $\p_x\left(\tau_y<\infty\right)=1$ for all $x,y\in \I$, and a diffusion which is not recurrent is said to be \emph{transient}. 
\end{enumerate}
\end{definition}

It is known (e.g. II.6 in \cite{BorSal02}) that the finiteness of a scale function at a boundary means that the corresponding end point is attractive, i.e. if $S(\alpha)=0$, then $\lim_{t\to\infty}X_t=\alpha$ with positive probability. This is closely related to the transiency of the diffusion, as seen in the next proposition.
\begin{proposition}[see Proposition 2.2 in \cite{Salminen85}]\label{prop trans1}
Let $X_t$ be as in Agreement \ref{def X}. Then $X_t$ is transient if and only if $S(\alpha)>-\infty$ and/or $S(\beta)<\infty$, i.e. at least one of the end points is attractive.
\end{proposition}

Another important feature of the scale function is its relation to the hitting time distribution of a diffusion (see e.g. II.4 in \cite{BorSal02}): for $\alpha<a<x<b<\beta$ we have
\begin{align}\label{eq hit}
\begin{aligned}
\p_x\left(\tau_a<\tau_b\right)&=\frac{S(b)-S(x)}{S(b)-S(a)},\quad\text{and }\\
\p_x\left(\tau_b<\tau_a\right)&=\frac{S(x)-S(a)}{S(b)-S(a)}.
\end{aligned}
\end{align}

Strictly speaking, \eqref{eq hit} tells us the distributions under the condition that we hit either $a$ or $b$ for a given $a<x<b$, but not whether a diffusion eventually exits from an interval $(a,b)$ almost surely. That it actually exits almost surely from an arbitrary interval $(a,b)$ with compact closure in $\I$ has been proved e.g. in \cite[Theorem 6.11]{Klebaner99} for It\^o diffusions. For completeness, we include here a proof for a general diffusion.

\begin{lemma}\label{lemma hit}
Let $X_t$ be as in Agreement \ref{def X} and let $x\in(z,y)$, where $\alpha<z<y<\beta$. Then $\p_x\left(\min\{\tau_z,\tau_y\}<\infty\right)=1$, i.e. $X_t$ exits from an open interval with a compact closure in finite time with probability $1$.
\end{lemma}

\begin{proof}
If $X_t$ is recurrent, the claim is clear. Therefore, assume that $X_t$ is transient. By II.20 in \cite{BorSal02} we can say that
\begin{align}\label{eq G0}
\p_x\left(\tau_y<\infty\right)=\frac{G_0(x,y)}{G_0(y,y)},
\end{align}
where $G_0$ is the Green function associated to a diffusion $X_t$ (e.g. II.11 in \cite{BorSal02}). Furthermore, by II.11 in \cite{BorSal02}, we have
\begin{align*}
G_0(x,y)=\lim_{\substack{a\to\alpha\\ b\to\beta}}\frac{\left(S(x)-S(a)\right)\left(S(b)-S(y)\right)}{\left(S(b)-S(a)\right)}.
\end{align*}
Using this expression we can rewrite \eqref{eq G0} as
\begin{subequations}
\begin{align}
\p_x\left(\tau_y<\infty\right)&=\lim_{a\to\alpha}\frac{S(x)-S(a)}{S(y)-S(a)}.\label{eq hit time y}\\
\intertext{Similarly}
\p_x\left(\tau_z<\infty\right)&=\lim_{b\to\beta}\frac{S(b)-S(x)}{S(b)-S(z)}.\label{eq hit time z}
\end{align}
\end{subequations}

It follows from Proposition \ref{prop trans1} that $X$ is transient if and only if $S(\alpha)>-\infty$ or $S(\beta)<\infty$. If $S(\alpha)=-\infty$, we have $\p_x\left(\tau_y<\infty\right)=1$. If $S(\beta)=\infty$, we have $\p_x\left(\tau_z<\infty\right)=1$. If both $S(\alpha)$ and $S(\beta)$ are finite, letting $z\to\alpha$ and $y\to\beta$ we obtain (recalling that $S(\alpha)=0$)
\begin{align*}
\lim_{y\to\beta}\p_x\left(\tau_y<\infty\right)=\frac{S(x)}{S(\beta)}\quad&\text{ and }\quad \lim_{z\to\alpha}\p_x\left(\tau_z<\infty\right)=1-\frac{S(x)}{S(\beta)}.
\end{align*}
As these sum up to $1$, we conclude that also in this case $\p_x\left(\min\{\tau_z,\tau_y\}<\infty\right)=1$.
\end{proof}

A consequence of Lemma \ref{lemma hit} is that for $\alpha<a<b<\beta$ the stopping time $\tau_{(a,b)}$ is finite a.s. and that for a transient diffusion $\p_x\left(\lim_{t\to\infty}X_t=\alpha \text{ or }\lim_{t\to\infty}X_t=\beta\right)=1$.

\subsection{Infinite values}

We identify simple conditions under which the variance is infinite.

\begin{proposition}\label{prop inf}
Assume that one of the following holds.
\begin{enumerate}[(i)]
	\item\label{item i} $\beta=\infty$ and $\lim_{b\to\infty}\p_x\left(\tau_b<\infty\right)b^2=\infty$.
	\item $\alpha=-\infty$ and $\lim_{a\to-\infty}\p_x\left(\tau_a<\infty\right)a^2=\infty$.
\end{enumerate}
Then $V(x)=\infty$. 
\end{proposition}

\begin{proof}
As the cases are analogous, we only prove the case (i). First notice that for every $x\in(a,b)\subset \mathcal{I}$
\begin{align}
\Var_x\left\{X_{\tau_{(a,b)}}\right\}&=(a-b)^2\p_x\left(\tau_a<\tau_b\right)\p_x\left(\tau_b<\tau_a\right).\label{eq formula}
\end{align}

We split the proof in two according to whether $\p_x\left(\tau_b<\infty\right)=1$ for all $b>x$, or not. 

\begin{enumerate}[1.]
	\item Assume that $\p_x\left(\tau_b<\infty\right)=1$ for all $b>x$. Then we know that $\p_x(\tau_\alpha<\infty)=\lim_{a\to\alpha}\p_x(\tau_a<\infty)=0$. If this would not be true, $\alpha$ would be attainable and hence either absorbing or killing in which case the process would not anymore enter the interior of $\I$ after hitting $\alpha$. Thus this would violate the assumed equality $\p_x\left(\tau_b<\infty\right)=1$ for all $b>x$. 
	
	Furthermore, $X_t$ hits $n>x$ with probability 1 and by Lemma \ref{lemma hit} we know that $X_t$ exits from every interval $(a,n)$, $x\in(a,n)$, with probability 1 so that $\p_x\left(\tau_{a}<\tau_n\right)+\p_x\left(\tau_{a}>\tau_n\right)=1$. Combining this into the fact that $\p_x\left(\tau_\alpha<\infty\right)=0$ lets us to choose a decreasing sequence $a_n<x$, in such a way that $\p_x\left(\tau_{a_n}<\tau_n\right)=\p_x\left(\tau_{a_n}>\tau_n\right)=\frac{1}{2}$ for all $n\in\N$, $n>x$. Substituting these into \eqref{eq formula} gives
	\begin{align*}
	\Var_x\left\{X_{\tau_{(a_n,n)}}\right\}&=(a_n-n)^2\frac{1}{4}.
	\end{align*}
	As we let $n\to\infty$, this tends to infinity.
	\item Assume now that there exists $b^*>x$ such that $\p_x\left(\tau_b<\infty\right)<1-\delta$ for some $\delta\in(0,1)$ and all $b>b^*$. If $\alpha>-\infty$ then 
	\begin{align*} 
\Var_x\left\{X_{\tau_{(\alpha,b)}}\right\}&
=(b-\alpha)^2\p_x\left(\tau_b<\infty\right)\left(1-\p_x\left(\tau_b<\infty\right)\right)\\
 &>\p_x\left(\tau_b<\infty\right)(b-\alpha)^2\delta\to\infty,\quad\text{as $b\to\infty$.}
	\end{align*}
If $\alpha=\infty$ then take a descending sequence $a_n\rightarrow -\infty$. Since the process $X$ exits every compact interval almost surely, then $P(\tau_{a_n}<\tau_{b^*})>\delta$ for every $n$, and $P(\tau_{a_n}<\tau_{b^*})\rightarrow P(\tau_{b^*}=\infty)$ as $n\rightarrow \infty$. Thus
\begin{align*}
\Var_x\left\{X_{(a_n,b^*)}\right\}&=(b^*-a_n)^2P(\tau_{a_n}<\tau_{b^*})(1-P(\tau_{a_n}<\tau_{b^*}))\\
&>(b^*-a_n)^2\delta(1-P(\tau_{a_n}<\tau_{b^*}))\rightarrow \infty,\quad\text{as $n\to\infty$.}\qedhere
\end{align*}
\end{enumerate}
\end{proof}

\begin{example}
Let $X_t$ be a Brownian motion on $(0,\infty)$ with killing at $0$. Then $S(x)=x$, $S(0)=0$, $S(\infty)=\infty$, and 
	\[\p_x(\tau_b<\infty)b^2=\frac{S(x)}{S(b)}b^2=xb\to\infty\quad \text{as }b\to\infty.\] 
	Hence $X_t$ satisfies the condition \eqref{item i} from Proposition \ref{prop inf}, and consequently $V(x)=\infty$.
\end{example}

Basically, Proposition \ref{prop inf} says that if a diffusion is too likely to travel too far towards an unbounded end point the achievable variances are unbounded. Especially, as a corollary we see that attractive unbounded end point always leads to infinite values. 

\begin{corollary}\label{cor inf}
Let $X_t$ be transient, and assume that at least one unbounded boundary point is attractive. Then $V(x)\equiv\infty$.
\end{corollary}

\begin{proof}
Assume $\beta$ is an unbounded, attractive endpoint. That is, $\beta=\infty$ and $S(\beta)<\infty$. Then $\p(\lim_{t\to\infty}X_t=\beta)=\delta>0$ and thus $\p_x\left(\tau_b<\infty\right)>\delta$ for all $b>x$. Therefore, 
$\lim_{b\to\infty}b^2\p_x\left(\tau_b<\infty\right)=\infty$, and the claim follows from Proposition \ref{prop inf}. The case for the boundary $\alpha$ is analogous.
\end{proof}

\subsection{Assumptions to get finite variance}

As the recurrent case is quite simple to handle, we need to assume the diffusion to be transient. Moreover, the inspection of transient diffusions falls naturally into three parts: either exactly one of the end point is attractive or both are. That is, we assume that one of the following assumptions hold in order to get a finite, interesting problem.

\begin{assumption}\label{assumption}
\begin{enumerate}[C{a}se (I)]
	\item Let $\alpha>-\infty$ and assume that $\alpha$ is attractive and $\beta$ is not (i.e. $S(\alpha)=0$ and $S(\beta)=\infty$) and that $\lim_{b\to\beta}\p_x\left(\tau_b<\infty\right)b^2=0$.
	\item Let $\beta<\infty$ and assume that $\beta$ is attractive and $\alpha$ is not (i.e. $S(\alpha)=-\infty$ and $S(\beta)<\infty$), and that $\lim_{a\to\alpha}\p_x\left(\tau_a<\infty\right)a^2=0$.
	\item Let $-\infty<\alpha<\beta<\infty$ and assume that both end points are attractive (i.e. $S(\alpha)=0$ and $S(\beta)<\infty$).
\end{enumerate}
\end{assumption}

Two cases not covered in the Proposition \ref{prop inf} or Assumption \ref{assumption} are the ones where $X_t$ is transient, and $\beta=\infty$ with $\lim_{b\to\infty}\p_x\left(\tau_b<\infty\right)b^2\in(0,\infty)$ and $\alpha=-\infty$ with $\lim_{a\to-\infty}\p_x\left(\tau_a<\infty\right)a^2\in(0,\infty)$. These special cases are discussed at Section \ref{sec special case} together with a killed diffusion case.

It is a normal sight in non-discounted problems that one needs some transiency in order to get interesting results. However, we would like to stress that the quadratic nature of the problem \eqref{eq prob} forces quite a strong transiency requirement for the finiteness: it is not enough that the process is transience, but it also needs to wander sufficiently rarely toward an unbounded end point.

Lastly, the following additional technical assumption will help to simplify the general result when facing randomized stopping times. 

\begin{assumption}\label{assumption 2}
\begin{enumerate}[(I)]
\item Let the conditions of Case (I) from Assumption \ref{assumption} hold. For each $c\in\I$ let $Z(c)$ denote the set of $z$ maximizing the ratio $\frac{z^2-\alpha^2-2c(z-\alpha)}{S(z)}$. For each $c\in\I$ for which $Z(c)$ has more than one element we assume that
\begin{equation}\label{eq:asumpineqI}
\E_{\hat{z}_i}\left\{X_{\tau_{(\alpha,\hat{z}_s)}}\right\}>c,
\end{equation}
where $\hat{z}_i=\inf\{Z(c)\}$ and $\hat{z}_s=\sup\{Z(c)\}$.\label{as2 Case I}
\item Let the conditions of Case (II) from Assumption \ref{assumption} hold. For each $c\in\I$ let $Y(c)$ denote the set of $y$ maximizing the ratio $\frac{y^2-\beta^2-2c(y-\beta)}{S(\beta)-S(y)}$. For each $c\in\I$ for which $Y(c)$ has more than one element we assume that
\begin{equation}\label{eq:asumpineqII}
\E_{\hat{y}_s}\left\{X_{\tau_{(\hat{y}_i,\beta)}}\right\}>c,
\end{equation}
where $\hat{y}_i=\inf\{Y(c)\}$ and $\hat{y}_s=\sup\{Y(c)\}$.\label{as2 Case II}
\end{enumerate}
\end{assumption}
Notice that $Z(c)$ contains $\inf\{Z(c)\}$ and $\sup\{Z(c)\}$ since the ratio $\frac{z^2-\alpha^2-2c(z-\alpha)}{S(z)}$ is continuous in $z$. Similar observation holds for $Y(c)$.

This assumption is indeed quite a techinal one. However, the sets $Y(c)$ and $Z(c)$ typically only contain one element, and when the sets contain more than one element the inequalities \eqref{eq:asumpineqI} and \eqref{eq:asumpineqII} are typically fulfilled. It seems that diffusions that do not satisfy Assumption \ref{assumption 2} are very marginal ones, and one has to carefully construct a specific diffusion in order to find a counter example that does not satisfy the assumption above (cf. example in Subsection \ref{subsec ex random}).
 
\section{Results}\label{sec results}

\subsection{Case (I): $\alpha>-\infty$ is attractive while $\beta$ is not}

In this case, by \eqref{eq hit time y}, $\p_x\left(\tau_b<\infty\right)=\frac{S(x)}{S(b)}$, and consequently the condition  $\lim_{b\to\beta}\p_x\left(\tau_b<\infty\right)b^2=0$ can be written as $\lim_{b\to\beta}\frac{b^2}{S(b)}=0$. Later in the section we show examples of diffusions satisfying these conditions.  Notice that $\beta$ can be either finite or infinite. 

The following theorem solves the variance problem under the Case (I) totally. We have three different designings for our main theorem, depending how general assumptions we make.

\begin{theorem}\label{theo main}
Let $X_t$ be as in Agreement \ref{def X} on $\I=(\alpha,\beta)$. Fix $x\in\I$ and assume that Assumption \ref{assumption}(I) holds. Then $V(x)<\infty$. Furthermore:
\begin{enumerate}[(A)]
\item\label{theo main Gen} There exist $a^*\in[\alpha,x]$ and $b^*,z^*\in[x,\beta)$ such that $\tau^*=\xi_{p^*}\tau_{(a^*,b^*)}+(1-\xi_{p^*})\tau_{(\alpha,z^*)}$ is an optimal stopping time, where $\xi_{p^*}$ is a Bernoulli random variable with a parameter $p^*\in[0,1]$. 
(For a solution algorithm and discussions, see Subsection \ref{subsec algo}.)
\item\label{theo main Res} Assume in addition that Assumption \ref{assumption 2}(I) holds. Then there exist $z_1^*,z_2^*\in[x,\beta)$ such that $\tau^*=\xi_{p^*}\tau_{(\alpha,z_1^*)}+(1-\xi_{p^*})\tau_{(\alpha,z_2^*)}$ is an optimal stopping time, where $\xi_{p^*}$ is a Bernoulli random variable with a parameter $p^*\in[0,1]$.

\item\label{theo main Sim} Assume in addition that $S(x)$ is differentiable and $\frac{S'(z)}{S(z)}(z-\alpha)$ is non-decreasing. Then $\tau_{(\alpha,z^*)}$ is an optimal stopping time, where $z^*=z^*(x)$ is the unique solution on $(x,\beta)$ to 
	\begin{align}\label{eq z}
	\frac{S(z)-S(x)}{\frac{1}{2}S(z)-S(x)}=\frac{S'(z)}{S(z)}(z-\alpha).
	\end{align}
	Furthermore, the value reads as $V(x)=(z^*-\alpha)^2\frac{S(x)}{S(z^*)}\left(1-\frac{S(x)}{S(z^*)}\right)$.
\end{enumerate}
\end{theorem}

\begin{example}
In this example, we show that there are diffusions satisfying the conditions of the main theorem. 
\begin{itemize}
	\item As will be shown in examples (Section \ref{sec examples}) Geometric Brownian motion $dX_t=\mu X_t dt +\sigma X_t dW_t$, with $\sigma>0$ and $\mu<-\frac{1}{2}\sigma^2$ satisfies the conditions Case (I) in Assumption \ref{assumption}. Furthermore, as in this case $S(x)=\frac{x^{1-\frac{2\mu}{\sigma^2}}}{1-\frac{2\mu}{\sigma^2}}$, it can be straightforwardly checked that also the the monotonicity condition in \eqref{theo main Sim} is satisfied.
	\item For the Bessel process (see \cite[p. 137]{BorSal02}) $dX_t = \frac{2v+1}{2X_t} dt + dW_t$, and for it $S(x)=-\frac{x^{-2v}}{2v}$. For $v<-1$, $\I=(0,\infty)$ with $0$ as an exit boundary. Furthermore, when $v< -1$ we have $S(0)=0$, $S(\infty)=\infty$, and $\p_x(\tau_b<\infty)b^2=x^{-2v}b^{2+2v}$ so that the conditions of Assumption \ref{assumption}(I) are satisfied. Moreover, $\frac{S'(z)}{S(z)}(z-\alpha)\equiv -2v$ is a constant and hence also the monotonicity condition in \eqref{theo main Sim} is satisfied.
	\item CIR-process (see Chapter 2 in \cite{Kuznetsov04}) $X_t$ satisfies the stochastic differential equation $dX_t=(a-bX_t)dt + \sigma\sqrt{X_t}dW_t$, and its scale function $S(x)=\int_0^x y^{-A-1}e^{\theta y}dy$, where $A=\frac{2a}{\sigma}-1$, $\theta=\frac{2b}{\sigma}$ and $\sigma>0$. The range $\I=(0,\infty)$. If we choose $b>0$ and $A\leq -1$, then the boundary $0$ is exit and the Case (I) in Assumption \ref{assumption} is satisfied. 
\end{itemize}
\end{example}

The monotonicity condition in \eqref{theo main Sim} may look quite peculiar, but in concrete examples it is usually satisfied. It should be mentioned, however, that for complicated scale functions (e.g. for logistic diffusion) its precise examination may be laborious. Moreover, as is seen from the theorem, this monotonicity condition simplifies considerably the solution as under it, the optimal stopping time is an ordinary hitting threshold rule of the form $\tau_{(\alpha,z)}$. 

The general case in part \eqref{theo main Gen} is more complex to handle. In that case one can find an optimal stopping time which is one of the following : a hitting time of the form $\tau_{(\alpha,z^*)}$, a randomization between two such times, a randomization between $\tau_{(\alpha,z^*)}$ and $0$ (stop immediately, in the theorem indicated by a stopping rule $\tau_{(x,x)}$), or a randomization between $\tau_{(\alpha,z^*)}$ and $\tau_{(a^*,b^*)}$, where $\alpha<a^*< x< b^*<\beta$. In literature, the optimality of the first three types have been reported before in variance stopping problems applying geometric L\'evy processes (see \cite{GadPed15}). However, the last type of randomization has not been reported explicitly before. We give an example of such a case in Subsection \ref{subsec ex random}.

Although in the general case Theorem \ref{theo main} offers no explicit solution, the proof of the general result provide us an algorithm how to find the solution. The algorithm is presented in Subsection \ref{subsec algo}. It is based on the division of the state space into two regions: One where the randomized solution is optimal and the other where the usual hitting time policy is optimal.

The actual proof for Theorem \ref{theo main} is given in the next section. In practice, there are two ways to prove theorem. One could lean heavily on the verifiation theorem and proceed in the lines of \cite{GadPed15} and \cite{Pedersen11}. Here we take an alternative route and utilize a game theory, showing that at least this particular problem class can quite surprisingly be seen as an application of zero-sum games.

\subsection{Case (II): $\beta<\infty$ is attractive while $\alpha$ is not}

In this case, by \eqref{eq hit time z}, $\p_x\left(\tau_a<\infty\right)=\frac{S(\beta)-S(x)}{S(\beta)-S(a)}$, and consequently the condition  $\lim_{a\to\alpha}\p_x\left(\tau_a<\infty\right)a^2=0$ can be written as $\lim_{a\to\alpha}\frac{a^2}{S(a)}=0$. Notice that $\alpha$ can be finite or $-\infty$.

\begin{theorem}\label{theo mainII}
Let $X_t$ be as in Agreement \ref{def X} on $\I=(\alpha,\beta)$. Fix $x\in \I$ and assume that the conditions of Case (II) in Assumption \ref{assumption} holds. Then $V(x)<\infty$. Furthermore:
\begin{enumerate}[(A)]
\item\label{theo main II Gen}  There exist $a^*,y^*\in(\alpha,x]$ and $b^*\in[x,\beta]$ such that $\tau^*=\xi_{p^*}\tau_{(y^*,\beta)}+(1-\xi_{p^*})\tau_{(a^*,b^*)}$ is an optimal stopping time, where $\xi_{p^*}$ is a Bernoulli random variable with a parameter $p^*\in[0,1]$. 
\item\label{theo main II Res} If in addition to Assumption \ref{assumption}(II) also Assumption \ref{assumption 2}(II) holds, then there exist $y_1^*,y_2^*\in(\alpha,x]$ such that $\tau^*=\xi_{p^*}\tau_{(y_1^*,\beta)}+(1-\xi_{p^*})\tau_{(y_2^*,\beta)}$ is an optimal stopping time, where $\xi_{p^*}$ is a Bernoulli random variable with a parameter $p^*\in[0,1]$.

\item\label{theo main II Sim} Assume in addition to Assumption \ref{assumption}(II) that $S(x)$ is differentiable and $\frac{S'(y)}{S(\beta)-S(y)}(\beta-y)$ is non-increasing. Then $\tau_{(y^*,\beta)}$ is an optimal stopping time, where $y^*=y^*(x)$ is the unique solution on $(\alpha,x)$ to 
	\begin{align}\label{eq y}
	\frac{S(x)-S(y)}{S(x)-\frac{1}{2}S(y)-\frac{1}{2}S(\beta)}=\frac{S'(y)}{S(\beta)-S(y)}(\beta-y).
	\end{align}
		Furthermore, the value reads as $V(x)=(\beta-y^*)^2\frac{S(\beta)-S(x)}{S(\beta)-S(y^*)}\left(\frac{S(x)-S(y^*)}{S(\beta)-S(y^*)}\right)$.
\end{enumerate}
\end{theorem}

\begin{example}
Let us show that there are diffusions satisfying the conditions of Theorem \ref{theo mainII}. To that end, let $X_t$ be a diffusion on $(0,\beta)$ that satisfies the condition of Theorem \ref{theo main}. Then the diffusion $Y_t=-X_t$ on $(-\beta,0)$ defined through $\hat{S}(y)=-S(-y)$, $\hat{m}(y)=m(-y)$, $Y_0= -x$ is a well defined diffusion that satisfies the conditions of Theorem \ref{theo mainII}.
\end{example}

\subsection{Case (III): $-\infty<\alpha<\beta<\infty$, both end points attractive}
As the state space is now finite, the value is always bounded with $\frac{1}{4}(\beta-\alpha)^2$ and is hence finite. Furthermore, the diffusion hits one or the other end point almost surely as we let $t\to\infty$, with $\p_x(\lim_{t\to\infty}X_t=\beta)=\frac{S(x)}{S(\beta)}=1-\p_x(\lim_{t\to\infty}X_t=\alpha)$. 

In this case, the solution reads as follows.

\begin{theorem}\label{theo mainIII}
Let $X_t$ be as in Agreement \ref{def X} on $\I=(\alpha,\beta)$. Fix $x\in\I$ and assume that the conditions of Case (III) in Assumption \ref{assumption} holds. Then $V(x)<\infty$. Furthermore:
\begin{enumerate}[(I)]
\item Assume that $\E_x\left\{X_{\tau_{(\alpha,\beta)}}\right\}\leq\frac{1}{2}(\beta+\alpha)$. Then the statements in Theorem \ref{theo main} hold true when requiring that $z^*$ in \ref{theo main}\eqref{theo main Sim} is either a unique solution to \eqref{eq z}, or, if the root does not exist, $z^*=\beta$. \label{theo main III 1}

\item Assume that $\E_x\left\{X_{\tau_{(\alpha,\beta)}}\right\}>\frac{1}{2}(\beta+\alpha)$. Then the statements in Theorem \ref{theo mainII} hold true when requiring that $y^*$ in \ref{theo mainII}\eqref{theo main II Sim} is either a unique solution to \eqref{eq y}, or, if the root does not exist, $y^*=\alpha$.\label{theo main III 2}
\end{enumerate}
\end{theorem}

In Case (III) both end points are attractive. Consequently, not only the stopping boundary, but also whether the stopping boundary is above or below the current state, depends on the initial state. Specially, we use $\tau_{(0,z)}$ if the upper boundary $\beta$ is "closer" to the initial state $x$ and $\tau_{(y,\beta)}$ if the lower boundary is "closer", and this "closeness" is measured with $\E_x\left\{X_{\tau_{(\alpha,\beta)}}\right\}$.

\begin{example}
Here we show examples of process satisfying Case (III) of Assumption \ref{assumption}. In principle, it is very effortless to construct these diffusions. If we truncate a diffusion at the interior of its initial domain, and stipulate killing on the new boundaries, we can make practically every diffusion eligible to Case (III). As an example, take Brownian motion on $\I=(\alpha,\beta)$, with $-\infty<\alpha<\beta<\infty$, and impose killing at both end points. Then clearly $S(\alpha)=\alpha>-\infty$ and $S(\beta)=\beta<\infty$, and we immediately have a process with both boundaries being attractive and hence satisfying Assumption \ref{assumption}(III). 
\end{example}

\section{Proof of Theorem \ref{theo main} (Case (I))}\label{sec proof}
We only need to prove the Case (I), as Cases (II)--(III) can be returned to that case, as will be demonstrated in Section \ref{sec cases II & III}. Throughout the section we assume that $X_t$ is as in Agreement \ref{def X}, and that the conditions of Case (I) in Assumption \ref{assumption} hold.

We prove Theorem \ref{theo main} (Case (I)) by leaning on known results from game theory. For any random variable $Y$ we have $\Var(Y)=\inf_c\E_x\left\{(Y-c)^2\right\}$. Hence our variance problem can be written in the form
\begin{align}\label{eq game}
\sup_\tau\Var_x\left\{X_\tau\right\}=\sup_\tau\E_x\left\{\left(X_\tau-\E_x\left\{X_\tau\right\}\right)^2\right\}=\sup_\tau\inf_c\E_x\left\{\left(X_\tau-c\right)^2\right\}=:\sup_\tau\inf_cA(\tau,c;x).
\end{align}
Here, $A(\tau,c;x)$ is strictly convex with respect to $c$ (being parabola), and consequently we can interpret the problem as an infinite, strictly convex two-player zero-sum game. Before diving into the actual proof, we need to settle some background results.

First of all, in our proof of Theorem \ref{theo main}, we assume that $\alpha=0$. This simplifies arguments and causes no loss of generality. Indeed, let $X_t$ be as in Agreement \ref{def X} with $\alpha\in(-\infty,\infty)$, and let $S(x)$ be the scale function and $m(x)$ the speed measure associated to $X_t$. We assume, like earlier, that $S(\alpha)=0$ and that $X_t$ is as in Agreement \ref{def X}. 
If $\alpha\neq0$, we define an auxiliary process $Y_t$ on a state space $(0,\beta-\alpha)$ by defining a scale function $\hat{S}$, a speed measure $\hat{m}$, and a starting point $Y_0$ through
\begin{align}\label{eq trans}
\hat{S}(y):=S(y+\alpha),\quad \hat{m}(y):=m(y+\alpha),\quad Y_0=y:=x-\alpha.
\end{align}
Then $Y$ is a well defined diffusion on $(0,\beta-\alpha)$ inheriting its boundary behaviour from $X_t$, and we can now study the optimal variance stopping problem for $Y$

The desired solution concerning the diffusion $X_t$ can be retrieved by inverting the transformations in \eqref{eq trans}.

\subsection{Preliminaries I --- useful facts about an auxiliary optimal stopping problem}\label{subsec aux}
In the way to solve our game \eqref{eq game}, we need some information on $A(\tau,c;x)$ as well. For this reason, in this subsection we present some knowledge on the auxiliary embedded quadratic optimal stopping problem
\begin{align}\label{eq associated}
V^c(x)=\sup_\tau\E_x\left\{(X_\tau-c)^2\right\}, 
\end{align}
where $c\in\I$. This is a classical optimal stopping problem, and it is partially solved in the following lemma.

\begin{lemma}\label{lemma stop}
Let Assumption \ref{assumption}(I) hold. Furthermore, let $c\in\I$, $c<\frac{1}{2}\beta$, and let $z_c$ be the greatest point that maximizes $\frac{z^2-2cz}{S(z)}$. For $x\leq z_c$, $\tau_{(0,z_c)}$ is an optimal stopping time to \eqref{eq associated} and the value reads as
\begin{align}\label{eq optimal stop}
V^c(x)=\frac{{z_c}^2-2cz_c}{S(z_c)}S(x)+c^2.
\end{align}
If $\beta<\infty$ and $c\geq \frac{1}{2}\beta$, then $\tau_{\{0\}}$ is an optimal stopping time.

In addition, for all $x\in\I$, an optimal stopping time is the hitting time $\tau_{D_c}=\inf\{t\geq0\mid X_t\in D_c\}=\tau_{(a,b)}$ for some $a\leq x\leq b$, where $D_c=\{x\mid V^c(x)=(x-c)^2\}$ is the stopping set associated to the embedded problem with a parameter $c$. 
\end{lemma}

Notice that the lemma solves the auxiliary problem explicitly only partially, for $x\in(0,z_c)$, while generally stating that for every $x\in\I$ the optimal stopping time is a hitting time without further knowledge on the stopping region $D_c$. This is usually enough and in the few exceptional cases where the stopping region $D_c$ is needed, it must be investigated individually for those cases.  

\begin{proof}[Proof of Lemma \ref{lemma stop}]
\begin{enumerate}[1.]
	\item Assume first that $c<\frac{1}{2}\beta$. As $\p_x(\tau_{(0,z)}<\infty)=\frac{S(x)}{S(z)}$ by \eqref{eq hit time y}, we have for an arbitrary $z\geq x$ 
\begin{align}
\E_x\left\{(X_{\tau_{(0,z)}}-c)^2\right\}&=\left(z-c\right)^2\p_x(\tau_{(0,z)}<\infty)+ c^2\left(1-\p_x(\tau_{(0,z)}<\infty)\right)\notag\\
&=\frac{z^2-2cz}{S(z)}S(x)+c^2. \label{eq Ezc}
\end{align}
As we have assumed that $\lim_{z\to\beta}\frac{z^2}{S(z)}=0$, the maximizer of the ratio $\frac{z^2-2cz}{S(z)}$ must be smaller than $\beta$. On the other hand, the ratio $\frac{z^2-2cz}{S(z)}$ is non-positive for all $z\leq 2c$ and positive for all $z>2c$. Therefore, we see that there must be at least one point that maximizes $\frac{z^2-2cz}{S(z)}$, and it is between $(2c,\beta)$. Let $z_c$ be the greatest of such points (which exists since $\frac{z^2-2cz}{S(z)}$ is $z$-continuous). Recalling that $X_\infty=0$ a.s., we have for all stopping times $\tau$
\begin{align*}
\E_x\left\{(X_{\tau}-c)^2\right\}&=\E_x\left\{X_\tau^2-2cX_\tau \right\}+c^2\\
&=\E_x\left\{\frac{X_\tau^2-2cX_\tau}{S(X_\tau)}S(X_\tau)\mid X_\tau\neq 0\right\}\p\left(X_\tau\neq 0\right)\\
&+\overbrace{\E_x\left\{X_\tau^2-2cX_\tau \mid X_\tau=0\right\}}^{=0}\p\left(X_\tau=0\right)+c^2\\
&\leq \frac{z_c^2-2cz_c}{S(z_c)}\E_x\left\{S(X_\tau)\right\}+c^2\leq \frac{z_c^2-2cz_c}{S(z_c)}S(x)+c^2,
\end{align*}
where the first inequality follows by the maximality of $z_c$ and the second one follows from the fact that $S(X_t)$ is a positive local martingale and hence a supermartingale.  
As the value $\frac{z_c^2-2cz_c}{S(z_c)}S(x)+c^2$ is attained by $\tau_{(0,z_c)}$ for all $x<z_c$, it is an optimal stopping time for all $x\leq z_c$. 
\item Assume now that $c\geq\frac{1}{2}\beta$. We can see that the stopping time $\tau_{\{0\}}$ gives a value $c^2$. As $\frac{z^2-2cz}{S(z)}$ is negative for all $z\in(0,2c)$, by \eqref{eq Ezc} the value would be smaller than $c^2$ if stopping at $(0,2c)$. Consequently $(0,2c)$ belongs to a continuation region and we see at once that if $c\geq\frac{1}{2}\beta$, the optimal stopping time is $\tau_{\{0\}}$, i.e stop at zero. This could also be interpreted as $\tau_{(0,\beta)}$, as $\beta$ is never reached.

\item Finally, let us prove the optimality of $\tau_{D_c}$ for every $x\in\I$.  

By items above it is known that $(0,2c)$ is in the continuation region and that $z_c\in D_c$. Moreover, for all stopping times $\tau$ and sequences $b_n$ such that $b_n\to\infty$ as $n\to\infty$ we have, by Case (I) of Assumption \ref{assumption}, that $\lim_{n\to\infty}\E_x\left\{(X_{\tau\land\tau_{b_n}}-c)^2\right\}=\E_x\left\{(X_{\tau}-c)^2\right\}$. It follows now from \cite[Theorem 6.3(III)]{LamZer13} that $\tau_{D_c}$ is an optimal stopping time. Furthermore, as $X_t$ is a continous process, $\tau_{D_c}$ can be written, for a given x, as an exit time in the form $\tau_{(a,b)}$ for some $a\leq x \leq b$.
 \qedhere
\end{enumerate}
\end{proof}

Although our approach is a game theoretic one, we nevertheless need the following verification theorem. It will be used in the final conclusion to verify that the optimal stopping solution found within our restricted game theoretic setting is also optimal among all admissible stopping times.

\begin{proposition}[Theorem 2.1 in \cite{Pedersen11}]\label{theorem Pedersen}
Assume that a constant $c^*=c^*(x)$ is such that the value function for an auxiliary problem 
\[V^{c^*}(x)=\sup_\tau\E_x\left\{\left(X_\tau-c^*\right)^2\right\}\]
is finite and the optimal stopping time $\tau^{c^*}$ that produces the value $V^{c^*}(x)$ satisfies the condition $c^*=\E_x\left\{X_{\tau^{c^*}}\right\}$. Then $\tau^{c^*}$ is also an optimal stopping time for the variance stopping problem \eqref{eq prob}. 
\end{proposition}

\subsection{Preliminaries II --- useful facts about zero-sum games}
In this subsection we present some useful facts from game theory, which help solving the game \eqref{eq game}.

\begin{definition}
\begin{enumerate}[(A)]
	\item We say that a game has a \emph{value $V$}, if 
\begin{align*}
\sup_\tau\inf_cA(\tau,c;x)=\inf_c\sup_\tau A(\tau,c;x).
\end{align*}
If the value exists, we can write $V=\inf_c\sup_\tau A(\tau,c;x)=\sup_\tau\inf_cA(\tau,c;x)$.
\item A \emph{pure strategy} is a strategy that uses a single stopping time $\tau$ or a state $c\in\I$.
\item A \emph{mixed strategy} is a strategy that mixes pure strategies using some known probability distribution (cf. randomized stopping times in Definition \ref{def random stopping time}). 
\item If a game has a value $V$ and there exists pure or mixed strategies $\tau^*$ and $c^*$ for which $A(\tau^*,c;x)\geq V$ for all $c$ and $A(\tau,c^*;x)\leq V$ for all $\tau$, we call $\tau^*$ and $c^*$ \emph{optimal strategies} for sup- and inf-player, respectively.
\item Let $c^*$ be an optimal strategy for the inf-player. We call a pure strategy $\tau$ \emph{essential}, if $V=A(\tau,c^*;x)$. 
\end{enumerate}
\end{definition}
We will only need mixed strategies of the form $\hat{\tau}:=\xi_p \tau_1+(1-\xi_p)\tau_2$, where $\xi_p$ is Bernoulli random variable with $p\in[0,1]$ (see Bernoulli randomized stopping time in Definition \ref{def Bernoulli}). Notice also that even if $\tau$ is essential, it is not necessarily an optimal strategy as we may have $\inf_c A(\tau,c;x)<A(\tau,c^*;x)$ (cf. example in Subsection \ref{subsec ex random}).

We have the following known result concerning the value and optimal strategies regarding infinite strictly convex zero-sum games on compact regions (Corollaries 2.2 and 2.3, and Section 5 in \cite{Bohnenblust50}, see also Theorem 4.3.1 in \cite{Karlin59}).

\begin{proposition}\label{theo Karlin} Assume that a payoff function $B(x,y):\mathcal{X}\times \mathcal{Y}\to \R$ of a game is continuous on both variables, that $B(x,y)$ is $y$ -strictly convex,  that the pure strategies $y$ of the inf-player takes value on compact connected set $\mathcal{Y}\subset \R$, and that the pure strategies $x$ of the sup-player takes value on compact, convex set $\mathcal{X}$ from $n$-dimensional Euclidean space. 
Then the game has a value using mixed strategies. Furthermore: 
\begin{enumerate}[(A)]
	\item the sup-player has an optimal mixed strategy involving at most $2$ pure strategies. That is, there is an optimal strategy of the form $\xi_px_1+(1-\xi_p)x_2$, where the sup-player applies a pure strategy $x_1\in\mathcal{X}$ with probability $p$ and $x_2\in\mathcal{X}$ with probability $1-p$. Moreover, these pure strategies $x_1$ and $x_2$ are essential strategies. 
	\item the inf-player has a unique optimal strategy that is a pure strategy $y^*\in\mathcal{Y}$ which minimizes $\sup_{X\in\mathcal{M}(\mathcal{X})}B(X,y)$, where $\mathcal{M}(\mathcal{X})$ is the set of all mixed strategies of the form $\xi_px_1+(1-\xi_p)x_2$, where $p\in[0,1]$ and $x_1$, $x_2\in\mathcal{X}$.
\end{enumerate}
\end{proposition}

Proposition \ref{theo Karlin} concerns games where the strategies of the sup-player takes values in a compact convex Euclidean space, whence \emph{a priori} the strategy set of the sup-player in our game \eqref{eq game} is the set of all stopping times. We shall overcome this discrepancy by writing stopping times as hitting times with two boundaries, as these boundaries take values from real line. This procedure will eventually allow us to utilize Proposition \ref{theo Karlin} for a restricted stopping problem where only hitting times are considered. 

Next we justify that the two games --- one with stopping time strategies from $\mathcal{T}$ and the other one with stopping boundary strategies from $\mathbb{R}^2$ -- are interchangeable.

Suppose we have a game given by a payoff $A_{\hat{\mathcal{T}}}(\tau,c;x)=\E_x\{f(X_\tau,c)\}$ where the pure strategies of the sup-player take values in $\hat{\mathcal{T}}=\{\tau_{(a,b)}|(a,b)\in \mathcal{X}\subset \mathbb{R}^2\}$ and the pure strategies of the inf-player take values in $\mathcal{Y}\subset\mathbb{R}$. Define another game by a payoff $A_{\mathcal{X}}((a,b),c;x)=\E_x\{f(X_{\tau_{(a,b)}},c)\}$, where the pure strategies of the sup-player take values in $\mathcal{X}\subset \mathbb{R}^2$ and the pure strategies of the inf-player take values in $\mathcal{Y}\subset\mathbb{R}$. It is quite obvious that the games $A_{\hat{\mathcal{T}}}$ and $A_\mathcal{X}$ are interchangeable with pure strategies. However, it may not be so obvious that they are equivalent also for mixed strategies mixing two pure strategies. This is proven in the next lemma.

\begin{lemma}
Let the payoff functions $A_{\hat{\mathcal{T}}}$ and $A_{\mathcal{X}}$ be as above and let the sets $\mathcal{X}\subset \R^2$ and $\mathcal{Y}\subset \R$. Furthermore, let $(a_i,b_i)\in\mathcal{X}$, $a_i\leq x\leq b_i$, for $i=1,2$, be pure strategies for the game $A_{\mathcal{X}}$ and let $p\in(0,1)$. Then with a mixed strategy $(A,B)=\xi_p(a_1,b_1)+(1-\xi_p)(a_2,b_2)$ we have $A_{\mathcal{X}}((A,B),c;x)=A_{\hat{\mathcal{T}}}(\tau_{(A,B)},c;x)$, where $\tau_{(A,B)}$ is a mixed stopping strategy $\xi_p\tau_{(a_1,b_1)}+(1-\xi_p)\tau_{(a_2,b_2)}$.
\end{lemma}

\begin{proof}
From conditional expectation we get:
\begin{align*}
A_{\mathcal{X}}((A,B),c;x)&=p\E_x\{f(X_{\tau_{(a_1,b_1)}},c)\}+(1-p)\E_x\{f(X_{\tau_{(a_2,b_2)}},c)\}\\
 &=A_{\hat{\mathcal{T}}}(\tau_{(A,B)},c;x)
\end{align*}
proving the claim.
\end{proof}


\subsection{Solving the game and proving the main theorem}\label{subsec solve game}
Recall our assumptions: we assume that Case (I) in Assumption \ref{assumption} holds and that $\alpha=0$. Let $x\in(0,\beta)$ be fixed.

Our approach for proving Theorem \ref{theo main} (Case (I)) is to use Proposition \ref{theo Karlin} on a two-player zero-sum game related to our problem. However, Proposition \ref{theo Karlin} only provide a solution for a game where strategies are restricted to a compact set. For that reason we define next a restricted two-player zero-sum game where strategies are restricted to certain compact sets. Then Proposition \ref{theo Karlin} will provide a solution, which will translate to a solution for a variance problem where variance is maximized over a certain restricted set of hitting times. By choosing the compact strategy sets wisely, the verification theorem can be used afterwards to verify that the this solution is also a solution to the initial variance problem.

First, we construct proper bounds to be used for the strategies of the restricted game. Analysis similarly to the proof of Lemma \ref{lemma stop}, shows that there exists finitely the greatest point that maximizes the ratio $\frac{u}{S(u)}$ for $u\geq x$. Call it $u^*$. That is, $u^*:= \sup\{u|\forall v \in [x,\beta): \tfrac{u}{s(u)}\geq \tfrac{v}{s(v)}\}<\beta$. Further, define 
\[M_x:=2\frac{u^*}{S(u^*)}S(x).\] 
We define the strategy set for the inf-player as the compact set $[0,M_x]$. 

Next we construct a proper compact strategy set of the form $[0,x]\times[x,B_x]$ for the sup-player. If $\beta<\infty$ we can choose $B_x=\beta$. So assume that $\beta=\infty$. Then we can define 
\[\bar{b}=\begin{cases}
3M_x,&M_x\geq x\\
3x,& M_x<x
\end{cases}\]
We can rewrite the pay-off function as
\begin{align}
A((a,b),c;x)&=(a-c)^2+\left((b-c)^2-(a-c)^2\right)\p_x(\tau_b<\tau_a)\notag\\
&=(a-c)^2+\frac{(b-c)^2-(a-c)^2}{S(b)-S(a)}\left(S(x)-S(a)\right).\label{eq Aabc}
\end{align}
Now, define 
\[\varepsilon:= \frac{(\bar{b}-M_x)^2-(\max(x^2,M_x^2)}{S(\bar{b})}>0.\]
Then from \eqref{eq Aabc} we know that for every given $a\in[0,x]$ and $c\in[0,M_x]$ we have $A(a,\bar{b},c;x)\geq (a-c)^2+\varepsilon\left(S(x)-S(a)\right)$. As $\frac{(b-c)^2-(a-c)^2}{S(b)-S(a)}$ approaches to zero as $b$ tends to $\beta$ by Assumption \ref{assumption}(I), there exists $\tilde{b}_x<\infty$ such that  $\frac{(b-c)^2-(a-c)^2}{S(b)-S(a)}<\varepsilon$ for all $b>\tilde{b}_x$ and for all $a\in[0,x]$ and $c\in[0,M_x]$. Hence, if $\beta=\infty$, we can choose $B_x:=\tilde{b}_x$. It follows that for all $a\in[0,x]$, $c\in[0,M_x]$, and $b>B_x$ we have
\begin{align*}
A((a,\bar{b}),c;x)&\geq(a-c)^2+\varepsilon\left(S(x)-S(a)\right)\geq (a-c)^2+\frac{(b-c)^2-(a-c)^2}{S(b)-S(a)}\left(S(x)-S(a)\right)\\
&=A((a,b),c;x),
\end{align*}
and there is equality only with the choice $a=x$. In other words, no optimal strategy can exceed $B_x$. Now, it holds from Lemma \ref{lemma stop} that both for $\beta<\infty$ and $\beta=\infty$ we have
$$\sup_{\tau\in\mathcal{T}}\EE_x\{(X_\tau-c)^2\}=\sup_{(a,b)\in [0,x]\times[x,B_x]}\EE_x\{(X_{\tau_{(a,b)}}-c)^2\}.$$
With this as our motivation we shall define the strategy of the sup-player as $[0,x]\times[x,B_x]$ in the restricted game.

\begin{descrip}\label{definition sets}
For $\beta=\infty$ and a fixed $x\in(0,\infty)$ we define a payoff function for a two-player zero-sum game through $A((a,b),c;x)=\E_x\left\{(X_{\tau{(a,b)}}-c^2)\right\}$. The pure strategies of the inf-player are $c\in[0,M_x]$, and the pure strategies of the sup-player are $(a,b)\in[0,x]\times [x,B_x]$. For $\beta<\infty$, we can choose $M_x=B_x=\beta$. As $x$ is fixed, we can ease the notations by defining $A((a,b),c):=A((a,b),c;x)$.

With these choices, our game is to find a value $V$ for the following game: 
\begin{align}\label{eq game restricted}
\sup_{(a,b)\in [0,x]\times [x,B_x]}\inf_{c\in[0,M_x]}A((a,b),c)=\inf_{c\in[0,M_x]}\sup_{(a,b)\in[0,x]\times [x,B_x]}A((a,b),c).
\end{align}
\end{descrip}

Now, we state our main existence result in the game theoretic framework. In Proposition \ref{prop game} that follows it, we find further characteristics on our game. Recall that $z_c$ is the greatest state that maximizes $\frac{z^2-2cz}{S(z)}$.

\begin{theorem}\label{theo main game}
Let $x\in(0,\beta)$ be fixed. Assume that Assumption \ref{assumption}(I) holds and consider the game of Description \ref{definition sets}. Then, there exist optimal strategies $c^*$ and $y^*$ such that 
\[V=A(y^*,c^*),\]
where, $c^*$ is the unique pure strategy on $[0,M_x]$ that minimizes 
\[\sup_{(a,b)\in[0,x]\times[x,B_x]}A((a,b),c).\]
Furthermore, $y^*$ is potentially a mixed strategy of the form $\xi_p(a^*,b^*)+(1-\xi_p)(0,z_{c^*})$, for some $(a^*,b^*)\in[0,x]\times[x,B_x]$, and where $\xi_p$ is a Bernoulli random variable with a parameter $p\in[0,1]$. Moreover, $(a^*,b^*)$ and $(0,z_{c^*})$ are essential strategies.
\end{theorem}

\begin{proof}
We can now readily apply Proposition \ref{theo Karlin} using compact strategy sets $\mathcal{Y}=[0,M_x]$ and $\mathcal{X}=[0,x]\times[x,B_x]$ together with a strictly $c$-convex payoff function $A((a,b),c)$. That $c^*$ can be reached as a minimizer for $\sup_{(a,b)\in[0,x]\times[x,B_x]}A((a,b),c)$ (using pure strategies rather than mixed) follows from the fact (Lemma 2.4.7 in \cite{Vor77}) that
$$\sup_{(a,b)\in[0,x]\times[x,B_x]} A((a,b),c)=\sup_{(a,b)\in\mathcal{M}([0,x]\times[x,B_x])} A((a,b),c),$$ 
where $\mathcal{M}([0,x]\times[x,B_x])$ denotes the mixed strategies on $[0,x]\times[x,B_x]$.

To show that one of the essential strategies given by Proposition \ref{theo Karlin} is of the form $(0,z)$, let us show that $c^*\geq\hat{c}$, where $\hat{c}$ is the smallest $c$ for which $z_c$ is greater than $x$ i.e. $\hat{c}:=\inf\{c\mid z_c>x\}$. 

Let $c<\hat{c}$. Then 
\begin{align*}
\sup_{(a,b)}A((a,b),c)\geq (x-c)^2>(x-\hat{c})^2.
\end{align*}
Here the first inequality follows from the fact that stopping immediately is an admissible strategy, and the second inequality from the fact that by the definition of $\hat{c}$ we have also $c<\hat{c}<x$ as $z_{\hat{c}}>2\hat{c}$. To deduce further, let $\tilde{x}\geq x$ be the starting point for which $\tilde{x}=z_{\hat{c}}$. Then,
\begin{align*}
(x-\hat{c})^2\geq (\tilde{x}-\hat{c})^2=\sup_{z}\E_{\tilde{x}}\left\{(X_{\tau_{(0,z)}}-\hat{c})^2\right\}=\sup_{(a,b)}\E_{\tilde{x}}\left\{(X_{\tau_{(a,b)}}-\hat{c})^2\right\},
\end{align*}
where the last equality follows from Lemma \ref{lemma stop}. As  $\sup_{(a,b)}\E_{x}\left\{(X_{\tau_{(a,b)}}-\hat{c})^2\right\}$ is $x$-increasing for $x>\hat{c}$, we lastly have
\begin{align*}
\sup_{(a,b)}\E_{\tilde{x}}\left\{(X_{\tau_{(a,b)}}-\hat{c})^2\right\}\geq \sup_{(a,b)}\E_{x}\left\{(X_{\tau_{(a,b)}}-\hat{c})^2\right\}=\sup_{(a,b)}A((a,b),\hat{c}).
\end{align*}
All in all, we have shown that when $c<\hat{c}$ 
\[\sup_{(a,b)}A((a,b),c)>\sup_{(a,b)}A((a,b),\hat{c}).\]
As $c^*$ is the unique $c$ that minimizes $\sup_{(a,b)}A((a,b),c)$, this indicates that the minimizer $c^*$ cannot be smaller than $\hat{c}$. It follows that $(0,z_{c^*})$ is an admissible essential strategy.
\end{proof}

Notice that Theorem \ref{theo main game} translates into 

$$V=\sup_{\tau\in \mathcal{M}(\mathcal{X})}\inf_{c \in [0,M_x]}\E_x\left\{(X_{\tau}-c)^2\right\}=\E_x\left\{(X_{\tau_{y^*}}-c^*)^2\right\},$$
where $\mathcal{X}=\{\tau_{(a,b)}\mid (a,b)\in[0,x]\times[x,B_x]\}$ and where $\mathcal{M}(\mathcal{X})=\{\xi_p \tau_1+(1-\xi_p)\tau_2\mid p\in[0,1],\ \tau_1,\tau_2\in\mathcal{X}\})$. 
This kind of translation between an optimal stopping problem and a game may be valuable also when working with more complex non-linear optimal stopping problems.

Next we will get a better inspection on the exact solution of our game. In the proofs to come, we will use an auxiliary restricted game, where the sup-player's strategy is only one-dimensional strategy $(0,z)$, $z\in[x,B_x]$. In this restricted setting the game with pure strategies becomes
\begin{align}\label{eq rest game}
 \sup_{z\in[x,B_x]}\inf_{c\in[\hat{c},M_x]}A((0,z),c).
\end{align} 
By Theorem \ref{theo main game} the optimal $c^*$ and at least one essential strategy for the sup-player of the initial game of Description \ref{definition sets} is found within the pure strategies of this restricted game. To see this, notice that for every $c\in[\hat{c},M_x]$ we have 
\[\sup_{(a,b)}A((a,b),c)=\sup_{z}A((0,z),c),\]
Hence, also the minimum over these two games are the same (cf. Proposition \ref{theo Karlin}).

The following lemma is a straight consequence from Theorem 2.12.5 from \cite{Vor77} concerning convex, continuous zero-sum games on a unit square. 

\begin{lemma}\label{lemma A}
Let Assumption \ref{assumption}(I) hold and let $x\in\I$ be fixed. Consider the restricted game \eqref{eq rest game}.
\begin{enumerate}[(A)]
	\item If $\hat{c}\leq c^*<M_x$, then there exists an essential strategy $(0,z')$ satisfying
\[A'_c((0,z'),c^*)\geq0.\]
\item If $\hat{c}<c^*<M_x$, then there exist essential strategies $(0,z')$ and $(0,z'')$ satisfying
\begin{align}\label{eq A}
A'_c((0,z'),c^*)\geq0\quad\text{and}\quad A'_c((0,z''),c^*)\leq0.
\end{align}
\end{enumerate}
\end{lemma}

\begin{proposition}\label{prop game}
Let the assumptions of the Theorem \ref{theo main game} hold and consider the game of Description \ref{definition sets}. Let $c^*$ be the optimal strategy of the inf-player. Then
\begin{enumerate}[(A)]
	\item\label{item c*} $c^*\in[\hat{c},M_x)$, where $\hat{c}:=\inf\{c>0\mid z_c>x\}$.
	\item\label{item pilkku} essential strategies $(a^*,b^*)$ and $(0,z_{c^*})$ from Theorem \ref{theo main game} satisfy 
	\begin{align}\label{eq conditions}
	\begin{aligned}
	A'_c((0,z_{c^*}),c^*)&\geq0 \text{ and }\\
	A'_c((a^*,b^*),c^*)&\leq0
	\end{aligned}
	\end{align}
	\item\label{item p} Let $a^*,b^*$, and $z_{c^*}$ as in item \eqref{item pilkku}. Then the sup-player has an optimal strategy where the strategies $(0,z_{c^*})$ and $(a^*,b^*)$ are utilized with probabilities $p$ an $1-p$, respectively, where $p$ is determined by
	\begin{align}\label{eq p condition}
	pA'_c((0,z_{c^*}),c^*)+(1-p)A'_c((a^*,b^*),c^*)=0.
	\end{align}
\end{enumerate}
\end{proposition}

\begin{proof}
\begin{enumerate}[(A)]
	\item Let us first prove that the inf-player chooses $c^*\in[0,M_x)$. In this part of the proof, we utilize the interplay between $A((a,b),c)$ and $A(\tau_{(a,b)},c)$ and choose $\mathcal{T}_\mathcal{X}=\{\tau_{(a,b)}\mid (a,b)\in [0,x]\times[x,B_x]\}$ and $\mathcal{M}(\mathcal{T}_\mathcal{X})=\{\xi_p \tau_1+(1-\xi_p)\tau_2\mid p\in[0,1],\ \tau_1,\tau_2\in\mathcal{T}_\mathcal{X}\})$. As has been demonstrated above, there exists $u^*$ maximizing $\frac{u}{S(u)}$ on $u\in[x,\beta)$. Now, for any $\tau\in\mathcal{M}(\mathcal{T}_\mathcal{X})$ 
\begin{align}\label{eq calculation u}
\begin{aligned}
\E_x\{X_{\tau}\}=\E_x\left\{\frac{X_{\tau}}{S(X_{\tau})}S(X_{\tau})\right\}&\leq \frac{u^*}{S(u^*)}\left(p\E_x\left\{S(X_{\tau_1})\right\}+(1-p)\E_x\left\{S(X_{\tau_2})\right\}\right)\\
&\leq\frac{u^*}{S(u^*)}S(x)=\frac{1}{2}M_x,
\end{aligned}
\end{align}
where the last inequality follows from the fact that as a positive local martingale $S(x)$ is a supermartingale. 

Assume that the sup-player chooses $\tau\in\mathcal{M}(\mathcal{T}_\mathcal{X})$. Then the inf-player is faced with the problem to minimize 
\[A(\tau,c)=\E_x\left\{X_{\tau}^2\right\}-2c\E_x\left\{X_{\tau}\right\}+c^2.\]
Now, the inf-player can only affect to the term $-2c\E_x\left\{X_{\tau}\right\}+c^2$. It is a parabola, and hence by choosing $c=0$ it vanishes and choosing $c=M_x$ it is non-negative as $M_x \geq 2\E_x\left\{X_{\tau}\right\}$ by \eqref{eq calculation u}. It follows that for any strategy $\tau$ announced by the sup-player, the inf-player rather chooses $c=0$ than $c=M_x$ and so the minimizer of $\sup_{\tau\in\mathcal{M}(\mathcal{T}_\mathcal{X})}A(\tau,c)$ must be on $[0,M_x)$. 

Lastly, in the proof of Theorem \ref{theo main game} we already showed that $c^*\geq\hat{c}$.

\item[\eqref{item pilkku}--\eqref{item p}] Similar results are given in Lemmas 2.12.2 and 2.12.3 and Theorem 2.12.5(3) in \cite{Vor77} for games where the sup-player has one-dimensional strategy set. However, those proofs work unaltered also in our two-dimensional case. Hence, all we need to show is that for the two essential strategies, $(a^*,b^*)$ and $(0,z_{c^*})$, the inequalities are as in \eqref{eq conditions}. 

Let us now study further the restricted game \eqref{eq rest game}, where the sup-player's strategy set is one-dimensional. We will show that in the notations of Lemma \ref{lemma A} we can choose $z'=z_{c^*}$, i.e. $A'_c((0,z_{c^*}),c^*)\geq0$. First of all, by Theorem \ref{theo main game} $c^*$ is unique and $(0,z_{c^*})$ is an essential strategy by Lemma \ref{lemma stop} and admissible as $c^*\geq \hat{c}$ by item \eqref{item c*}. 

Suppose, contrary to our claim, that $A'_c((0,z_{c^*}),c^*)<0$. Then by Lemma \ref{lemma A} there must be $z_2<z_{c^*}$ maximizing $\frac{z^2-2c^*z}{S(z)}$, for which $A'_c((0,z_2),c^*)\geq0$, as only these kind of stopping thresholds are essential strategies for the game. We can certainly assume that $A'_c((0,z_2),c^*)=0$. Indeed, as $A'_c((0,z_2),c^*)=2(c^*-\frac{z_2}{S(z_2)}S(x))$, we can choose $\tilde{x}\geq x$ so that $2(c^*-\frac{z_2}{S(z_2)}S(\tilde{x}))=0$, whence $c^*$ is still the optimal strategy for the inf-player, $A'_c((0,z_{c^*}),c^*)<0$, and both $(0,z_2)$ and $(0,z_{c^*})$ are essential strategies. Hence, let $A'_c((0,z_2),c^*)=0$ and $A'_c((0,z_{c^*}),c^*)<0$, which translates to
\[\frac{z_2}{S(z_2)}S(x)=c^*<\frac{z_{c^*}}{S(z_{c^*})}S(x)\quad \Longleftrightarrow \quad \frac{z_2}{S(z_2)}<\frac{z_{c^*}}{S(z_{c^*})}.\]

We can now make the following deduction
\begin{align*}
A((0,z_2),c^*)&=A((0,z_{c^*}),c^*)=\frac{z_{c^*}}{S(z_{c^*})}\left(z_{c^*}-2c^*\right)S(x)+(c^*)^2\\
&>\frac{z_{2}}{S(z_{2})}\left(z_{2}-2c^*\right)S(x)+(c^*)^2=A((0,z_2),c^*),
\end{align*}
which is impossible. It follows that we must have $A'_c((0,z_{c^*}),c^*)\geq0$. \qedhere
\end{enumerate}
\end{proof}

\begin{remark}\label{remark game}
A few remarks about Proposition \ref{prop game} above.
\begin{enumerate}[1.]
	\item The item \eqref{item c*} guarantees that the restricted strategy set for the inf-player is not too restricted. Specifically, that $c^*$ is in the interior of $[0,M_x]$ means that both inequalities in \eqref{eq conditions} in the item \eqref{item pilkku} are satisfied for some essential stopping times.
	\item We see from items \eqref{item pilkku}--\eqref{item p} that in order to find an optimal solution, we need to find essential strategies $(0,z^*)$ and $(a^*,b^*)$ satisfying conditions \eqref{eq conditions}. Moreover, it is quite clear from items \eqref{item pilkku}--\eqref{item p}, that if $A'_c((a,b),c^*)=0$ for some essential strategy $(a,b)$, then this essential pure strategy is also an optimal one. 
	\item From the proof we can deduce that if $c^*>\hat{c}$, then by Lemma \ref{lemma A} the both essential strategies needed for conditions \eqref{eq conditions} can be chosen to be of the simple hitting time form $(0,z)$.
\end{enumerate}
\end{remark}

Lastly, we use the verification theorem \ref{theorem Pedersen} to verify that the optimal solution of the sup-player for the game of Description \ref{definition sets} with restricted compact strategy sets is optimal also for the variance stopping problem \eqref{eq prob} among all randomized stopping times.

\begin{lemma}\label{lemma random}
Let assumptions of Theorem \ref{theo main game} hold and let $c^*$, $a^*$, $b^*$, $z_{c^*}$, and $p$ be as in Proposition \ref{prop game}. Then the randomized stopping time $\tau^*:=\xi_p\tau_{(0,z_{c^*})}+(1-\xi_{p})\tau_{(a^*,b^*)}$ is an optimal stopping time to the variance stopping problem \eqref{eq prob} among all admissible randomized stopping times $\mathcal{T}$.
\end{lemma}

\begin{proof}
 We will show that $\tau^*$ and $c^*$ satisfy the conditions from the verification theorem (Proposition \ref{theorem Pedersen}). Let $y^*= \xi_p(0,z^*)+(1-\xi_{p})(a^*,b^*)$. As $A'_c(y^*,c)=2(c-\E_x\left\{X_{\tau_{y^*}}\right\})$ and $\E_x\left\{X_{\tau_{y^*}}\right\}=\E_x\left\{X_{\tau^*}\right\}$, the condition \eqref{eq p condition} can be written as
\[E_x\left\{X_{\tau^*}\right\}=c^*.\]
As $y^*$ is a randomization between two essential strategies, we have
\begin{align*}
\E_x\{(X_{\tau^*}-c^*)^2\}&=A(y^*,c^*)=A((a^*,b^*),c^*)\\
&=\sup_{(a,b)\in[0,x]\times[x,B_x]}A((a,b),c^*)\\
&=\sup_{(a,b)\in[0,x]\times[x,\beta)}A((a,b),c^*)\\
&=\sup_{\tau\in\mathcal{T}}\E_x\{(X_{\tau}-c^*)^2\}
\end{align*}
where the second to last equation follows from the choice of $B_x$ and the last equation from the fact that for every $c$ there is an optimal stopping time is of the form $\tau_{(a,b)}$ where $(a,b)\in [0,x]\times [x,B_x]$ by Lemma \ref{lemma stop}. It follows that both conditions of Proposition \ref{theorem Pedersen} are satisfied, and the claim follows. 
\end{proof}


We are now ready to conclude our proof of our main Theorem \ref{theo main}.

\begin{proof}[Proof of Theorem \ref{theo main}] \noindent

\begin{itemize}
	\item[Theorem \ref{theo main}\eqref{theo main Gen}] 
	Lemma \ref{lemma random} proves the claim.
	\item[Theorem \ref{theo main}\eqref{theo main Res}] As was noticed in Remark \ref{remark game}, if $c^*>\hat{c}$ then the two needed essential strategies for the conditions in \eqref{eq conditions} are both of the form $(0,z)$. For this reason, let us now prove that this condition follows from additional Assumption \ref{assumption 2}(I). This is done by showing that $A((0,z_{\hat{c}}),\hat{c})<0$, which excludes the possibility that $\hat{c}$ would be optimal, as by Lemma \ref{lemma A} $A((0,z_{c^*}),c^*)\geq0$. 

Now, either $z_{\hat{c}}=x$ or $z_{\hat{c}}>x$. In the former case we have 
\[\E_x\left\{X_{\tau_{(0,z_{\hat{c}})}}\right\}=x=z_{\hat{c}}>2\hat{c}>\hat{c}\] meaning that $A'_c((0,z_{\hat{c}}),\hat{c})=2\left(\hat{c}-\E_x\left\{X_{(0,z_{\hat{c}})}\right\}\right)<0$.

So, assume $z_{\hat{c}}>x$. This is possible only if for $\hat{c}$ there exists more than one state that maximizes $\frac{z^2-2\hat{c}z}{S(z)}$. Let $z_i$ be the smallest of them. By the definition of $\hat{c}$, we can immediately deduce that $x\in[z_i,z_{\hat{c}})$.

But now, Assumption \ref{assumption 2}(I) gives $\E_{z_i}\left\{X_{\tau_{\hat{c}}}\right\}=\frac{z_{\hat{c}}}{S(z_{\hat{c}})}S(z_i)>\hat{c}$, so that by monotonicity of $S(x)$ also $\E_{x}\left\{X_{\tau_{z_{\hat{c}}}}\right\}=\frac{z_{\hat{c}}}{S(z_{\hat{c}})}S(x)>\hat{c}$ for all $x\in[z_i,z_{\hat{c}})$. As was demonstrated above, in our case $x\in[z_i,z_{\hat{c}})$ meaning that again $A'((0,z_{\hat{c}}),\hat{c})<0$, and the claim follows.

\item[Theorem \ref{theo main}\eqref{theo main Sim}]
By Theorem \ref{theo main game} and Proposition \ref{prop game}\eqref{item pilkku} for a unique optimal strategy $c^*$ played by the inf-player, there exists a corresponding essential strategy $(0,z_{c^*})$. Now we will show that if $\frac{S'(z)}{S(z)}z$ is non-decreasing, the essential strategy $(0,z_{c^*})$ satisfies the condition $A'_c((0,z_{c^*}),c^*)=0$, indicating it to be an optimal strategy for the sup-player. 
	
\begin{enumerate}[1.]
		\item Let $z_c$ be the greatest point that maximizes $\frac{z^2-2cz}{S(z)}$. It is clear from Lemma \ref{lemma stop} that there is at least one such maximizer. Now will show that under the assumed monotonicity and differentiable conditions, it is the only one. 
		                                                   
By straight derivation, the first order optimality condition for the ratio $\frac{z^2-2cz}{S(z)}$ is
\begin{align}\label{eq c cond}
\frac{d}{dz}\frac{z^2-2cz}{S(z)}=0\quad\Longleftrightarrow\quad \frac{z-c}{\frac{1}{2}z-c}=\frac{S'(z)}{S(z)}z.
\end{align}
As was noticed in the proof of Lemma \ref{lemma stop}, the maximum point of the ratio $\frac{z^2-2cz}{S(z)}$ is attained on $(2c,\beta)$. Moreover, it can be easily checked that since $c>0$, the ratio $\frac{z-c}{\frac{1}{2}z-c}$ is $z$-decreasing and positive for $z>2c$. As we assumed the positive mapping $\frac{S'(z)}{S(z)}z$ to be non-decreasing, we see that for any $c\in\I$ there is at most one $z_c>2c$ satisfying the first order optimality condition \eqref{eq c cond}. Consequently, for each $c$ there exists exactly one $z_c$ maximizing $\frac{z^2-2cz}{S(z)}$. From \eqref{eq c cond} we also see that $z_c$ must be $c$-continuous under the stated monotonicity condition.

\item Let again $\hat{c}=\inf\{c\mid z_c>x\}$. As in this case $z_c$ is unique, $c$-continuous and increasing, we must have 
\[\E_x\left\{X_\tau{(0,z_{\hat{c}})}\right\}=x=z_{\hat{c}}>2\hat{c}>\hat{c}.\]
Especially this means 
\[A'_c((0,z_{\hat{c}}),\hat{c})=2\left(\hat{c}-\E_x\left\{X_\tau{(0,z_{\hat{c}})}\right\}\right)<0.\]
Thus we cannot have $c^*=\hat{c}$, as by Proposition \ref{prop game}\eqref{item pilkku} $A'_c((0,z_{c^*}),c^*)\geq0$. It follows that $c^*\in(\hat{c},M_x)$ and hence by Lemma \ref{lemma A} we must have two essential strategies $\tau_{(0,z')}$ and $\tau_{(0,z'')}$ satisfying \eqref{eq A}. However, as the essential strategy $\tau_{(0,z_{c^*})}$ is unique under the stated assumptions, we can conclude that $z'\equiv z''$ and that $A'_c((0,z_{c^*}),c^*)=0$ proving the optimality of the pair $((0,z_{c^*}),c^*)$. 

\item To actually find the unique pair $(c^*,z_{c^*})$, let $(0,z)$ be given for some $z\in[x,B_x]$. Then the inf-player wants to minimize
\[A((0,z),c)=\E_x\left\{X_{\tau_{(0,z)}}^2\right\}-2c\E_x\left\{X_{\tau_{(0,z)}}\right\}+c^2\]
with respect to $c$. This is easily seen to happen at $c=\E_x\left\{X_{\tau_{(0,z)}}\right\}=\frac{S(x)}{S(z)}z$. Substituting this into the game, the sup-player is left to maximize 
\[\sup_{z}A\left((0,z),\tfrac{S(x)}{S(z)}z\right)=\sup_z \left\{\frac{z^2-2z^2\frac{S(x)}{S(z)}}{S(z)}S(x)+\frac{S(x)^2z^2}{S(z)^2}\right\}.\]
By straight differentiation we get the first order optimality condition
\begin{align}\label{eq optimal}
\frac{d}{dz}A\left((0,z),\tfrac{S(x)}{S(z)}z\right)=0 \quad \Longleftrightarrow\quad \frac{S(z)-S(x)}{\frac{1}{2}S(z)-S(x)}=\frac{S'(z)}{S(z)}z
\end{align}
It is easy to check that $\frac{S(z)-S(x)}{\frac{1}{2}S(z)-S(x)}$ is negative for all $z<S^{-1}\left(2S(x)\right)$ and $z$-decreasing and positive for all $z>S^{-1}\left(2S(x)\right)$. As $\frac{S'(z)}{S(z)}z$ is positive and assumed to be non-decreasing, we see that there is at most one solution to \eqref{eq optimal}. Consequently, the unique solution on $(S^{-1}\left(2S(x)\right),\beta)$ to \eqref{eq optimal}, is the optimal stopping boundary $z^*$, the stopping time $\tau_{(0,z^*)}$ is an optimal stopping time to the problem \eqref{eq prob} and the value reads as $A((0,z^*),\frac{S(x)}{S(z^*)}z^*)$. \qedhere
\end{enumerate}
\end{itemize}
\end{proof}

The game theoretic proof for the existence of a value (Proposition \ref{theo Karlin}) also offers a way to identify the solution(cf. Sections 2.11-2.12 in \cite{Vor77} and Chapter 4 in \cite{Karlin59}), and this is presented in Subsection \ref{subsec algo}. 

\subsection{On optimality}\label{subsec on optimality}

\subsubsection*{Time-inconsistency and static and dynamic optimalities}

The variance stopping problem \eqref{eq prob} depends on the current state $x$. Hence it is also expected that also the stopping boundary $z^*(x)$ is highly sensitive on the starting point $x$ (see e.g. \eqref{eq z}). This phenomenon can be named \emph{time-inconsistency}, meaning that when the process moves from the initial value $x$ in $t$ units of time to a new location $y$ it yields different value and stopping rule (cf. \cite{PedPes12,Christensen17,Christensen18}). In sharp contrast to this, in the usual linear optimal stopping problem, \emph{time-consistent} problem, the stopping boundary is unaltered by the path of the underlying process. This naturally raises some questions whether our starting point dependent optimality (called \emph{static optimality} in \cite{PedPes12}) is adequate for all practical purposes. This question has been asked in \cite{PedPes12} where also a new optimality class, \emph{dynamic optimality}, was introduced in a mean-variance setting. The distinction between these two optimalities is discussed in details in \cite[Section 4]{PedPes12}. 

In short, the statically optimal strategy depends on the initial point, and hence it remembers the past. In contrast, the dynamically optimal strategy is independent from the initial state, and hence it ignores the past and is only pointing to the future. It is acquired by basically solving infinitely many optimal stopping problems dynamically in time where each new position of $X_t$ yields a new optimal stopping problem. In practice, in one-boundary cases, the dynamic optimal stopping threshold in the mean-variance setting is the solution to the equality $z^*(x)=x$. For the static and dynamic optimality in a mean-variance portfolio selection problem, see \cite{PedPes13}. In a pure variance stopping problem, as the one considered in this paper, the whole received value when stopping the process comes from the realized path of "past" (i.e. realized variance), whereas there are no value for the "present". This means that in pure variance problem it is dynamically never optimal to stop, as the value without the "past" is always zero. For a concrete example, see geometric Brownian motion example in Subsection \ref{subsec gbm}.

In the mean-variance setting both optimalities --- static as well as dynamic --- has sound interpretations. As explained in \cite{PedPes12}, the static optimality can be seen a sound strategy for a "pre-committed" investor, who evaluate his strategy at initial position, and does not re-evaluate the optimality criterion at later times. The dynamic optimality is a sound system for a "dynamic investor", who is non-committed to the initial strategy and re-evaluates his optimality criterion at each new time point. Although the dynamic optimality is sensible in a mean-variance setting, it nevertheless is not meaningful when maximizing variance alone, where stopping immediately in any given starting point is always unprofitable. I.e. in a pure variance stopping problem, it cannot be dynamically optimal to stop at any time. This rational solution, as the variance is realized through evolved path, and this path is dismissed totally by the dynamic optimality which only looks ahead. Observe, however, that solving dynamic optimal stopping boundary in practice requires knowledge on static optimal stopping boundary as well. In this way, Theorem \ref{theo main} might offer some basis also for dynamic optimality in more complex problem settings.

Worth noticing is that the usual, time-consistent linear optimal stopping problems can be seen to be simultaneously both dynamical as well as static in their nature. Indeed, in those problems a decision maker can decide his stopping strategy based only on the starting point $x$. On the other hand, he arrives to a same strategy also by re-evaluating his strategy at each new time point as the underlying process develops. This underlines the differenceses between linear and non-linear optimal stopping problems.

\subsubsection*{Randomized stopping times and time-inconsistency}
Recently time-inconsistent stopping problems have gained growing attention. For example in \cite{PedPes12,PedPes13,Gad15,Gad15b, GadPed15,Christensen17,Christensen18} the time-inconsistency is attained by considering a value function that is non-linear with respect to the expectation. In \cite{Huang17, Huang18, Huang2018} it is shown that a time-inconsistency is the outcome when one considers a more complex discount function. For a short literature summary on time-inconsistency in financial economics, see \cite{Christensen17}. In very recent papers \cite{Christensen17, Christensen18} Christensen and Lindensj\"o consider general time-inconsistent optimal stopping problems in game-theoretic framework. In their setup, they formulate the stopping problem as a game, which is played between uncountable amount of agents, one for each possible starting point of the underlying process. The solution to this game is then found as a possibly randomized equilibrium stopping time. This equilibrium differs slightly in nature from static and dynamic optimality. For example in \cite{Christensen18} it is shown how the equilibrium solution in geometric Brownian motion case differs from static and dynamic optimality in variance stopping problem and in mean-variance setting it coincides with dynamic optimality, but only with certain parameters. Applying this approach, Christensen and Lindesj\"o characterize the equilibrium of the problem and find necessary and sufficient conditions for it. Especially it is shown in these papers, and also for example in \cite{GadPed15}, how the solution in the general setting in the time-inconsistent problems cannot be reached without utilizing the randomized stopping time (or mixed strategy) -concept. As usual time-consistent stopping problems always involve only pure optimal stopping rules (see e.g. \cite{PesShi06,Shi78}), it follows that inconsistency of time is the needed trigger for randomized solutions to appear. It should be mentioned, however, that in the linear optimal stopping problems the decision maker can be indifferent about optimal stopping time; i.e. two different stopping times might both yield the same value. Then also a mixture of these two stopping times would give the value, but the difference to the time-inconsistent case is that this kind of randomization do not carry any additional value. 

The natural question is: why time-inconsistency offers a possibility to a randomized solution? One possible answer to this might be the additional demand posed for an optimal solution. To make this argument more clear, let us consider the verification theorem presented for variance problem (Proposition \ref{theorem Pedersen}). In it, we see how the stopping time is optimal only if it satisfies a requirement $\E_x\left\{X_{\tau^*}\right\}=c^*$ in addition to the standard optimality demand $V^c(x)=\E_x\left\{(X_{\tau^*}-c^*)^2\right\}$. It can be that this additional expectation requirement is simply "too much" to acquire with pure stopping time strategies, and hence enlarged class of stopping times is needed to solve the problem in general. 

However, it could be that the reverse question ought to be the right one: Why linear problem setting kills the need for a randomization? It can be argued that very generally seen, randomized stopping times are always the right class to consider, but time-homogenous environment adds smoothness, linearity, and some form of simplicity, which lead to pure optimal strategies. From this view, the linear "usual optimal stopping problems" are in fact the unusual ones.

\section{Proofs of Cases (II) and (III)}\label{sec cases II & III}

\subsection{Proof of Theorem \ref{theo mainII} --- Case (II)}
The proof of this theorem can be returned to Case (I) in the following way. Let $S(x)$ be the scale function and $m(x)$ the speed measure associated to $X_t$. 
If $\beta\neq0$, we can define an auxiliary process $Y_t$ on a state space $(\alpha-\beta,0)$ by defining the scale function $\hat{S}$, speed measure $\hat{m}$, and starting point $Y_0$ through
\begin{align}\label{eq trans2}
\hat{S}(y):=S(y+\beta),\quad \hat{m}(y):=m(y+\beta),\quad Y_0=y:=x-\beta.
\end{align}
Then $Y$ is well defined diffusion on $(\alpha-\beta,0)$ inheriting its boundary behaviour from $X_t$.

After this we can define another auxiliary process $Z_t$ on a state space $(0,\beta-\alpha)$ by defining the scale function $\check{S}$, speed measure $\check{m}$, and starting point $Z_0$ through
\begin{align}\label{eq trans3}
\check{S}(z):=-\hat{S}(-z)+\hat{S}(0),\quad \check{m}(z):=m(-z),\quad Z_0=z:=-y.
\end{align}
Then $Z$ is well defined diffusion on $(0,\beta-\alpha)$, with lower end point inheriting its behaviour from the upper end point of $Y$ and vice versa, and its scale function vanishing at the lower boundary: $\check{S}(0)=0$. It follows that 
the optimal variance stopping problem
\begin{align*}
\check{V}(z)=\sup_\tau\Var_z\left\{Z_\tau\right\},\quad Z_0=z
\end{align*}
can be solved utilizing Theorem \ref{theo main}. Consequently, the desired result concerning the diffusion $X_t$ on $(\alpha,\beta)$ can be retrieved from this by inverting the transformations in \eqref{eq trans2} and \eqref{eq trans3}.

\subsection{Proof of Theorem \ref{theo mainIII} --- Case (III)}

We need to consider only the case $\alpha=0$. If $\alpha\neq0$, we can make the same transformation we did in Case (I) to retrieve the case $\alpha=0$.

The main difference to the other cases is the fact that the type of the optimal solution depends on the location of the starting point $x$. This phenomenon arises when solving the embedded quadratic problem \eqref{eq associated}: The type of the solution is different depending on whether $c\leq \frac{1}{2}\beta$ or not, as we prove in the following lemma.

\begin{lemma}\label{lemma stop2}
Let $z_c$ be the greatest point on $[0,\beta]$ that maximizes $\frac{z^2-2cz}{S(z)}$, and $y_c$ the smallest point on $[0,\beta]$ that maximizes $\frac{y^2-2cy-\beta^2+2c\beta}{S(\beta)-S(y)}$. 
\begin{enumerate}[(A)]
	\item Assume that $c\leq \frac{1}{2}\beta$. Then $y_c=\beta$ and $z_c\in(2c,\beta]$ and, for all $x<z_c$, the optimal stopping time to the embedded quadratic problem \eqref{eq associated} is $\tau_{(0,z_c)}$ and the value reads as
\[V^c(x)=\frac{{z_c}^2-2cz_c}{S(z_c)}S(x)+c^2.\]
\item Assume that $c>\frac{1}{2}\beta$. Then $y_c\in[0,2c-\beta)$ and $z_c=0$ and, for all $x>y_c$, the optimal stopping time to the embedded quadratic problem \eqref{eq associated} is $\tau_{(y_c,\beta)}$ and the value reads as
\[V^c(x)=\frac{{y_c}^2-2cy_c-\beta^2+2c\beta}{S(\beta)-S(y_c)}S(x)+(\beta-c)^2.\]
\end{enumerate}
\end{lemma}

\begin{proof}
\begin{enumerate}[(A)]
	\item It is easily seen that when $c\leq \frac{1}{2}\beta$, then $\frac{z^2-2cz}{S(z)}$ is positive on $(2c,\beta]$ and negative elsewhere in the state space. On the other hand $\frac{y^2-2cy-\beta^2+2c\beta}{S(\beta)-S(y)}$ is negative on $[0,\beta)$. Therefore $y_c=\beta$ (the value with a stopping rule $\tau_{(y,\beta)}$ is maximized when $y=\beta$) and $z_c\in(2c,\beta]$. The proof for the solution for all $x<z_c$ is analogous to Lemma \ref{lemma stop}.
	\item It is easily seen that when $c>\frac{1}{2}\beta$, then $\frac{z^2-2cz}{S(z)}$ is negative on $(0,\beta]$. On the other hand $\frac{y^2-2cy-\beta^2+2c\beta}{S(\beta)-S(y)}$ is positive on $[0,2c-\beta)$ and negative elsewhere in the state space. Therefore $y_c\in[0,2c-\beta)$ and $z_c=0$. 
	
	Let $x>y_c$. Then for all stopping times $\tau$ we have
	\begin{align*}
	\E_x\left\{(X_\tau-c)^2\right\}&=\E_x\left\{\frac{X_\tau^2-2cX_\tau-\beta^2+2c\beta}{S(\beta)-S(X_\tau)}\left(S(\beta)-S(X_\tau)\right)\right\} +(\beta-c)^2\\
	&\leq \frac{y^2-2cy-\beta^2+2c\beta}{S(\beta)-S(y)}\E_x\left\{S(\beta)-S(X_\tau)\right\}+(\beta-c)^2\\
	&=\frac{y^2-2cy-\beta^2+2c\beta}{S(\beta)-S(y)}\left(S(\beta)-S(x)\right)+(\beta-c)^2,
	\end{align*}
where the first inequality follows by the maximality of $y_c$ and the second one follows from the fact that $S(X_t)$ is a bounded local martingale on $\I$ and hence martingale.

As this value is attained with $\tau_{(y_c,\beta)}$ we know that it must be the optimal stopping time.\qedhere
\end{enumerate}
\end{proof}

 The rest of the proof is analogous to Case (I) and Case (II). One only needs to do separately the cases $\E_x\left\{X_{\tau_{(0,\beta)}}\right\}\leq \frac{1}{2}\beta$ and $\E_x\left\{X_{\tau_{(0,\beta)}}\right\}> \frac{1}{2}\beta$.

\section{Special cases}\label{sec special case}
For the sake of completeness let us study here briefly the special cases which are not yet covered.

\subsection{Recurrent case}

In the recurrent case we have $-S(\alpha)=S(\beta)=\infty$.

\begin{lemma}\label{lemma recurrent}
Let $X_t$ be recurrent. 
\begin{enumerate}[(A)]
	\item Assume that $\alpha=-\infty$ or $\beta=\infty$. Then $V(x)\equiv\infty$.
	\item Assume that $-\infty<\alpha<\beta<\infty$. Then the optimal stopping time is $\tau_{(\alpha,\beta)}$ and the value reads $V(x)=\frac{1}{4}(\beta-\alpha)^2$.
\end{enumerate}
\end{lemma}

\begin{proof}
\begin{enumerate}[(A)]
	\item Let $\beta=\infty$. As $X_t$ is recurrent, we have $\p_x\left(\tau_b<\infty\right)=1$ for all $b>x$, and hence the claim follows straightly from Proposition \ref{prop inf}. The case $\alpha=-\infty$ is analogous.
	\item 
First note that an arbitrary random variable $Y$ on an interval $[\alpha,\beta]$ has the highest possible variance if $\PP(Y=\alpha)=\PP(Y=\beta)=\tfrac{1}{2}$. In this case $\Var(Y)=\tfrac{1}{4}(\alpha-\beta)^2$. Since $X_{\tau}$ takes values on $(\alpha,\beta)$ we must have $V(x)\leq \tfrac{1}{4}(\alpha-\beta)^2$. Let us next show that also the reversed inequality holds.

As $X$ is recurrent we can choose sequences $a_n$ and $b_n$ in such a way that $a_n\rightarrow \alpha$ and $b_n\rightarrow \beta$ as $n\rightarrow \infty$, and that $\p_x(X_{\tau_{(a_n,b_n)}}=a_n)=\p_x(X_{\tau_{(a_n,b_n)}}=b_n)=\tfrac{1}{2}$ for all $n\in\N$. 
To show that these sequences exist, let $x$ be, for simplicity, such that $S(x)=0$. Choose $a_n<x$ to be any decreasing sequence for which $\lim_{n\to\infty}a_n=\alpha$ and choose $b_n$ to satisfy $S(b_n)=-S(a_n)$. Then the sequences $a_n$ and $b_n$ satisfy required properties since
\begin{align*}
\p_x\left(\tau_{b_n}<\tau_{a_n}\right)=\frac{-S(a_n)}{S(b_n)-S(a_n)}=\frac{1}{\frac{S(b_n)}{-S(a_n)}+1}=\frac{1}{2}\quad \text{for all }n\in\N.
\end{align*}

Thus $\Var_x(X_{\tau_{(a_n,b_n)}})=\tfrac{1}{4}(a_n-b_n)^2\rightarrow \tfrac{1}{4}(\alpha-\beta)^2$ as $n\rightarrow \infty$. Since we have $V(x)\geq \Var_x(X_{\tau_{(a_n,b_n)}})$ for all $n\in\N$, we must also have $V(x)\geq \tfrac{1}{4}(\alpha-\beta)^2$ proving the claim. \qedhere
	\end{enumerate}
\end{proof}
The result states intuitively clear fact of how, in recurrent case, we should use the whole span of the state space. Notice that the optimal stopping time $\tau_{(\alpha,\beta)}$ is infinite almost surely. However, in the proof we saw that for every $\varepsilon>0$ there exists $a$ and $b$ such that $\tau_{(a,b)}$ is an almost surely finite $\varepsilon$-stopping time.

\subsection{Transient case with $\lim_{b\to\infty}\p_x\left(\tau_b<\infty\right)b^2\in(0,\infty)$}

Let us consider briefly Case (I) of Assumption \ref{assumption} with the condition that $\beta=\infty$ and $\lim_{b\to\infty}\p_x\left(\tau_b<\infty\right)b^2=\lim_{b\to\infty}\frac{b^2}{S(b)}\in(0,\infty)$.

It can be shown, mimicking the proof of Lemma \ref{lemma stop}, that also in this special case the optimal stopping time to an embedded quadratic problem
\begin{align*}
V^c(x)=\sup_\tau\E_x\left\{(X_\tau-c)^2\right\}
\end{align*}
is $\tau_{(\alpha,z_c)}$ for all $c\in\I$ and $x<z_c$. However, the main difference is that now we may have $z_c=\infty$, and $\tau_{(\alpha,\infty)}$ is unattainable in finite time almost surely. Nevertheless, we can write $z_c=\argmax\left\{\frac{z^2-2z c}{S(z)}\right\}$, and the value $V^c(x)=CS(x)+c^2<\infty$, where $C=\sup_z\left\{\frac{z^2-2z c}{S(z)}\right\}$. Unfortunately, general existence proofs of Theorem \ref{theo main} cannot be utilized straightforwardly as they require $z_c$ to be finite. However, applying the $c$-convexity and the known fact that $\varepsilon$-optimal strategies do exist in a game setting, we could modify our Theorem \ref{theo main game} to work also in this case. Moreover, the proof when $\frac{S'(x)}{S(x)}x$ is non-decreasing, Theorem \ref{theo main}\eqref{theo main Sim} can be quite straightforwardly modified to work also in this case. Summarizing, the following result holds.

\begin{lemma}\label{lemma special}
Let $X_t$ be as in Agreement \ref{def X} on $\I=(\alpha,\infty)$. Let Assumption \ref{assumption}(I) hold. We fix $x\in\I$ and assume that $\lim_{b\to\infty}\p_x\left(\tau_b<\infty\right)b^2=\lim_{b\to\infty}\frac{b^2}{S(b)}\in(0,\infty)$.
\begin{enumerate}[(A)]
\item The value exists and there exists $c^*>0$ such that $V(x)=\sup_\tau\E_x\left\{(X_\tau-c^*)^2\right\}$. Moreover, there exists an $\varepsilon$-optimal stopping time of the form $\tau_{(0,z)}$, or possible a randomization between stopping times $\tau_{(\alpha,z)}$ and $\tau_{(a,b)}$ for some $a,b,z\in[\alpha,\infty]$.
	\item If $S$ is differentiable and $\frac{S'(z)}{S(z)}(z-\alpha)$ is non-decreasing, then the value exists and the optimal stopping time to the optimal variance stopping problem \eqref{eq prob} is $\tau_{(\alpha,z^*)}$,	where $z^*$ is either the unique solution to 
	\begin{align*}
	\frac{S(z)-S(x)}{\frac{1}{2}S(z)-S(x)}=\frac{S'(z)}{S(z)}(z-\alpha),
	\end{align*}
	or, if the root does not exist, $z^*=\infty$.	Furthermore, the value reads as $V(x)=(z^*-\alpha)^2\frac{S(x)}{S(z^*)}\left(1-\frac{S(x)}{S(z^*)}\right)$.\label{lemma hard A}
	\end{enumerate}
 \end{lemma}

	In item \eqref{lemma hard A} whenever $z^*=\infty$, the value $V(x)$ is understood as a limit $V(x)=\lim_{z\to\infty}(z^*-\alpha)^2\frac{S(x)}{S(z^*)}\left(1-\frac{S(x)}{S(z^*)}\right)$.

\subsection{Diffusions with a constant killing}\label{subsec kill}

In this paper we have only considered unkilled diffusions. In this subsection we discuss a bit about a process with a constant killing. The main reason the killing makes the problem more difficult is because it introduces two-sided optimal stopping rules to the embedded auxiliary problem \eqref{eq associated}.

\subsubsection*{Problem setting}\label{subsub kill 1}
Consider a unkilled diffusion $X_t$ on $\I=(\alpha,\beta)$, associated with scale function $S(x)$ and speed measure $m(x)$. To fix ideas, let the boundaries be natural. Fix $x\in\I$. Let $\tilde{X}_t$ be otherwise similarly defined process, but with a constant killing rate $\lambda>0$. In practice, we can then interpret $\tilde{X}_t$ to be killed at an exponential rate, i.e.
\[\tilde{X}_t=\begin{cases}
X_t,&t<\zeta\\
\partial, & t\geq \zeta,
\end{cases}
\]
	where $\zeta\sim\text{Exp}(\lambda)$ is the life time of the diffusion $\tilde{X}_t$ and $\partial\notin\I$ a cemetery state, where $\tilde{X}_t$ is sent immediately when killed. (See for example Section X.4 in \cite{Dynkin65} or Chapter III in \cite{BluGet68}.)

Now, to consider our optimal variance stopping problem with killed process, we need to separate two cases: In the first the terminal time $\zeta$ is not observable and hence cannot be used as a stopping time. In the second the terminal time $\zeta$ is observable and it can be used as a stopping time giving a value corresponding to $X_{\zeta-}$. We can, similar to Section \ref{sec proof}, modify the original problem into a zero-sum game. In the first one, where $\zeta$ is not observable, we can make the following modifications:
\begin{align*}
V_1(x)=\sup_\tau\text{Var}_x(\tilde{X}_\tau)&
=\sup_\tau\inf_c\E_x\left\{e^{-\lambda\tau}({X}_\tau-c)^2\right\}=:\sup_\tau\inf_cA_1(\tau,c;x),
\end{align*}
where $\tau$ is a $\mathbbm{F}$-stopping time and $c\in\I$. In the second one, where $\zeta$ is observable, the game turns out to be
\begin{align*}
V_2(x)&=\sup_\tau\text{Var}_x(\tilde{X}_\tau)\\
&=\sup_\tau\inf_c \left\{\lambda\left(R_\lambda(x-c)^2\right)(x)+\E_x\left\{e^{-\lambda\tau}(X_\tau-c)^2-\lambda e^{-\lambda\tau}\left(R_\lambda(x-c)^2\right)(X_\tau)\right\}\right\},
\end{align*}
where $\left(R_\lambda(x-c)^2\right)(x)=\E_x\left\{\int_0^\infty e^{-\lambda t}(X_t-c)^2dt\right\}$ is the resolvent for the mapping $(x-c)^2$

Here $V_1$ corresponds to a case, where the decision maker receives nothing if the process dies before he had taken action. That is, for a killed process the variation is zero. On the other hand $V_2$ corresponds to a case, where the decision maker receives the total variation of the process during its lifetime if he has not taken action before the process dies. That is, for the killed process the variation is its whole lifespan. 

\subsubsection*{Variance stopping}

Let us consider from now on only the problem $V_1$. Consideration for the problem $V_2$ is analogous but more technical. The problem type is different from unkilled version in that now the killing, or "discounting", makes it easier to reach finite value. To get a finite value, we have to impose an assumption 
\begin{align}\label{eq finite}
\lim_{t\to\infty}\E_x\left\{e^{-\lambda t}X_t^2\right\}=0.
\end{align} 
The auxiliary embedded quadratic optimal stopping problem reads now as
\begin{align}\label{eq aux kill}
V_1^c(x)=\sup_\tau\E_x\left\{e^{-\lambda\tau}(X_\tau-c)^2\right\}.
\end{align}

It can be shown that for a given $c\in\I$ and $x\in\I$ it has a solution $\tau_{(a^*,b^*)}$ for some $a^*\in[\alpha,x]$ and $b^*\in[x,\beta]$ (cf. \cite{LamZer13}). The difference in unkilled problem is that now the optimal solution is typically two-boundary stopping rule and one-boundary rule rarely gives the value. Proceeding as in Section \ref{sec proof} we could again find $M_x$ and $B_x$ such that the game with pure strategies is
\begin{align}\label{eq kill game}
\sup_{(a,b)\in[0,x]\times [x,B_x]}\inf_{c\in[0,M_x]}A_1(\tau_{(a,b)},c;x).
\end{align}
This has a solution by Proposition \ref{theo Karlin}, and hence we could (if going through all the details) make a statement similar to Theorem \ref{theo main}(A), the only difference being that now the optimal stopping time to the problem would be $\tau^*=\xi_{p^*}\tau_{(a_1,b_1)}+(1-\xi_{p^*})\tau_{(a_2,b_2)}$ with $(a_i,b_i)\in[\alpha,x]\times [x,\beta]$, $i=1,2$, so that no one-boundary solution is known to be essential.

In this way, the general result concerning (constantly) killed diffusion would be attainable by applying existing procedure from Section \ref{sec proof}. However, the details of the solution would be quite laborious to clarify: could the solution be one-boundary solution? When is the solution a randomized stopping time? How many essential strategies there are?, etc. Hence killed diffusion solution is out of the scope of the present study. 
	
\section{Examples}\label{sec examples}

\subsection{Solution algorithm}\label{subsec algo}

Before proceeding to our examples, let us introduce a solution algorithm how to find the solution for all $x\in\mathcal{I}$. The algorithm is written for Case (I): We assume that $\alpha=0$, $\beta>0$, where $\alpha$ is attractive while $\beta$ is not and $\frac{b^2}{S(b)}\to0$ as $b\to\beta$. If $S$ is differentiable and $\frac{S'(x)}{S(x)}x$ is non-decreasing, the solution is easy to find applying Theorem \ref{theo main}\eqref{theo main Sim}. So, we assume now that the above mentioned mapping is not non-decreasing. 

Under these assumptions the solution is potentially a randomized solution, and there is no explicit way to tell what is the optimal stopping time. However, we can construct an algorithm based on the fact that the solution exists and is either a stopping time $\tau_{(0,z)}$ for some $z$ or a randomization between stopping times $\tau_{(0,z)}$ and $\tau_{(a,b)}$, where $0\leq a\leq x\leq b\leq z$. In the algorithm we separate these two cases. 

\begin{enumerate}[Step 1.]
\item \begin{enumerate}[(i)]
	\item For each $c\in\mathcal{I}$, solve the embedded quadratic problem 
	\begin{align}\label{eq as}
	V^c(x)=\sup_z E_x\left\{(X_{\tau_{(0,z)}}-c)^2\right\}.
	\end{align}
	A threshold $z_c$ that maximizes the ratio $\frac{z^2-2cz}{S(z)}$ is the maximizer for this embedded quadratic problem (cf. Lemma \ref{lemma stop}). Let 
	\[C=\left\{c\in\I\mid \exists \text{ multiply $z_{c}$ maximizing \eqref{eq as}}\right\}\] be the set of all $c$, for which there exists more than one maximizer for \eqref{eq as}. Typically, the set $C$ is finite.
	\item Take $c\in C$. Let $\mathcal{Z}_c=\left\{z\mid z=\argmax\{\frac{z^2-2cz}{S(z)}\}\right\}$ be the set of maximizers of \eqref{eq as} for $c$, and denote by $\underline{z}_{c}:=\inf\left\{\mathcal{Z}_c\right\}$ and $\overline{z}_{c}:=\sup\left\{\mathcal{Z}_c\right\}$ the smallest and greatest of such points.	Check whether assumption \ref{assumption 2}\eqref{as2 Case I} holds or not, i.e. is the condition $\E_{\underline{z}_c}\left\{X_{\tau_{(0,\overline{z}_c)}}\right\}>c$ met.
\end{enumerate}
	\item[Step 2a.] Assumption \ref{assumption 2}\eqref{as2 Case I} holds.
	\begin{enumerate}[(i)]
	\item For the chosen $c\in C$, define a randomized stopping time $\hat{\tau}_c(p):=\xi_p\tau_{(0,\underline{z}_{c})}+(1-\xi_p)\tau_{(\alpha,\overline{z}_{c})}$ with $\xi_p$ being a Bernoulli random variable with a parameter $p$, and define $\underline{x}_c$ and $\overline{x}_c$ to be the smaller and greater, respectively, solutions to the equations
	\begin{align*}
	\E_{x}\left\{X_{\hat{\tau}_c(1)}\right\}&=c \quad \Longleftrightarrow \quad \underline{x}_c=S^{-1}\left(\frac{c}{\underline{z}_c}S(\underline{z}_c)\right)\\ 
		\E_{x}\left\{X_{\hat{\tau}_c(0)}\right\}&=c \quad \Longleftrightarrow \quad \overline{x}_c=S^{-1}\left(\frac{c}{\overline{z}_c}S(\overline{z}_c)\right).
		\end{align*}
	Then for all $x\in(\underline{x}_c,\overline{x}_c)$ there exists a unique $p^*(x)\in(0,1)$ satisfying the condition $\E_x\left\{X_{\hat{\tau}_c(p^*(x))}\right\}=c$.
\item Repeat the step (i) for all $c\in C$ which satisfies Assumption \ref{assumption 2}\eqref{as2 Case I}. 
\item Define $\mathcal{J}_A:=\bigcup_{c\in C}(\underline{x}_c,\overline{x}_c)$ to be the set of points $x$ for which Assumption \ref{assumption 2}\eqref{as2 Case I} is satisfied.
\end{enumerate}\label{algo 2}

\item[Step 2b.] Assumption \ref{assumption 2}\eqref{as2 Case I} does not hold.
\begin{enumerate}[(i)]
	\item For the chosen $c\in C$, $\overline{x}_c$ and $\underline{x}_c$ are again the solutions to 
	\begin{align*}
	\E_{x}\left\{X_{\tau_{(0,\underline{z}_c)}}\right\}&=c \quad \Longleftrightarrow \quad \underline{x}_c=S^{-1}\left(\frac{c}{\underline{z}_c}S(\underline{z}_c)\right).\\
	\E_{x}\left\{X_{\tau_{(0,\overline{z}_c)}}\right\}&=c \quad \Longleftrightarrow \quad \overline{x}_c=S^{-1}\left(\frac{c}{\overline{z}_c}S(\overline{z}_c)\right).
	\end{align*}
\item	For $x\in(\underline{x}_c,\underline{z}_c]$, we can use a randomized stopping time $\hat{\tau}_c(p)=\xi_p\tau_{(0,\underline{z}_{c})}+(1-\xi_p)\tau_{(\alpha,\overline{z}_{c})}$. However, for $x\in(\underline{z}_c,\overline{x}_c)$ we need to define a randomized stopping time $\check{\tau}_c(p):=\xi_p\tau_{D_c}+(1-\xi_p)\tau_{(\alpha,\overline{z}_{c})}$. Here $\tau_{D_c}$ is an optimal stopping time, where $D_c$ is the stopping set for an embedded problem with a parameter $c$.
	\item Repeat the steps (i) -- (ii) for all $c\in C$, which does not satisfy Assumption \ref{assumption 2}\eqref{as2 Case I}.
	\item Define $\mathcal{J}_{\not{A}}:=\bigcup_{c\in C}(\underline{x}_c,\overline{x}_c)$ to be the set of points $x$ for which Assumption \ref{assumption 2}\eqref{as2 Case I} is not satisfied.
\end{enumerate}
\setcounter{enumi}{2}
\item The following is an optimal stopping time:
	\begin{align*}
	\begin{cases}
	\tau_{(\alpha,z^*(x))},\quad  &x\in\I\setminus\left(\mathcal{J}_{A}\cup\mathcal{J}_{\not{A}}\right)\\
	\hat{\tau}_c(x)=\xi_{p^*_x}\tau_{(0,\underline{z}_c)}+(1-\xi_{p^*_x})\tau_{(0,\overline{z}_c)},\quad &x\in\mathcal{J}_A\\
	\hat{\tau}_c(x)=\xi_{p^*_x}\tau_{(0,\underline{z}_c)}+(1-\xi_{p^*_x})\tau_{(0,\overline{z}_c)},\quad &x\in\mathcal{J}_{\not{A}}\,\&\, x\in(\underline{x}_c,\underline{z}_c]\\
		\check{\tau}_c(x)=\xi_{p^*_x}\tau_{D_c}+(1-\xi_{p^*_x})\tau_{(0,\overline{z}_c)},\quad &x\in\mathcal{J}_{\not{A}}\,\&\, x\in(\underline{z}_c,\overline{x}_c).
	\end{cases}
	\end{align*}
	First, when $x\in\I\setminus\left(\mathcal{J}_A\cup\mathcal{J}_{\not{A}}\right)$, then $z^*(x)$ is a maximizer of $(z-\alpha)^2\frac{S(x)}{S(z)}\left(1-\frac{S(x)}{S(z)}\right)$ and hence, if $S$ is differentiable, a solution (not necessarily unique, as there may be local extreme points!) to the first order optimality condition  \[\frac{S(z)-S(x)}{\frac{1}{2}S(z)-S(x)}=\frac{S'(z)}{S(z)}z.\]

Second, when $x\in\mathcal{J}_A$, then $x\in(\underline{x}_c,\overline{x}_c)$ for some $c$. The points $\underline{z}_c$ and $\overline{z}_c$ are maximizers associated with this $c$ and $p^*_x$ can be solved from \[
\E_x\left\{X_{\hat{\tau}(p^*_x)}\right\}=c\quad \Longleftrightarrow \quad p^*_x=\frac{\displaystyle\frac{c}{S(x)}-\frac{\overline{z}_c}{S(\overline{z}_c)}}{\displaystyle\frac{\underline{z}_c}{S(\underline{z}_c)}-\frac{\overline{z}_c}{S(\overline{z}_c)}}.
\]

Lastly, when $x\in\mathcal{J}_{\not{A}}$, then $x\in(\underline{x}_c,\overline{x}_c)$ for some $c$. One then needs to find $D_c$ associated with this $c$ and $p^*_c$ can be solved from  
\[
\E_x\left\{X_{\xi_{p^*_x}\tau_{D_c}+(1-\xi_{p^*_x})\tau_{(0,\overline{z}_c)}}\right\}=c.
\]
	\label{algo 3}
\item The value reads as
\begin{align*}
V(x)=\inf_c V^c(x)=\begin{cases}
(z^*(x)-\alpha)^2\frac{S(x)}{S(z^*)}\left(1-\frac{S(x)}{S(z^*)}\right),\quad&x\in\I\setminus\left(\mathcal{J}_{A}\cup\mathcal{J}_{\not{A}}\right)\\
\frac{\overline{z}_c^2-2c\overline{z}_c}{S(\overline{z}_c)}S(x)+c^2,\quad & x\in(\underline{x}_c,\overline{x}_c)\subset \mathcal{J}_A\cup\mathcal{J}_{\not{A}}.
\end{cases}
\end{align*}
Here $z^*(x)$ is as in Step \ref{algo 3}, and $c\in C$, $\underline{z}_c$, and $\overline{z}_c$ are the constants associated with the interval $(\underline{x}_c,\overline{x}_c)$.
\label{algo 4}
\end{enumerate}

In the algorithm we first identify the regions in which the solution is a randomized stopping time solution, after which we know that in everywhere else, a familiar threshold stopping time is an optimal one. 

There are three observations to make. First, in Step \ref{algo 3}, when $x\in\I\setminus\left(\mathcal{J}_{A}\cup\mathcal{J}_{\not{A}}\right)$, the optimal stopping threshold $z^*(x)$ is a solution to the first order optimality condition, but now as $\frac{S'(z)}{S(z)}z$ is not non-decreasing, it is not necessarily a unique solution. Therefore, one needs to check which solution is the maximizer. Second observation is that the value for $x\in(\underline{x}_c,\overline{x}_c)$, given in Step \ref{algo 4}, can be written with a constant stopping boundary $\overline{z}_c$ (or equivalently with $\underline{z}_c$). The reason for this is that the value of the variance stopping problem equals to the value of the embedded quadratic problem, and for $x\in(\underline{x}_c,\overline{x}_c)$ the corresponding constant $c$ is unaltered. 

Lastly, if Assumption \ref{assumption 2}\eqref{as2 Case I} does not hold, we see that Step 2 differs quite remarkably, as we do not know how $D_c$ looks like in general case. However, for $x\in(\underline{x}_c,\underline{z}_c]$ we are in the safe waters and we can randomize between $\tau_{(0,\underline{z}_c)}$ and $\tau_{(0,\overline{z}_c)}$. This follows from the fact that as $x<\underline{z}_c$, we can apply Lemma \ref{lemma A} to conclude that $\tau_{(0,\underline{z}_c)}$ and $\tau_{(0,\overline{z}_c)}$ are two essential strategies satisfying the conditions \eqref{eq A}.

\subsection{Geometric Brownian motion}\label{subsec gbm}
Let us first illustrate our results with geometric Brownian motion (which is also one of the examples considered in \cite{Pedersen11}). 

Now the state space is $\I=(0,\infty)$ and diffusion is a solution to the stochastic differential equation
\[
dX_t=\mu X_tdt+\sigma X_tdW_t,\quad X_0=x,\]
where $\mu\in\R$ and $\sigma\in\R_+$ are given coefficients. The scale function is given by
\[S(x)=\begin{cases}
\frac{x^{1-\frac{2\mu}{\sigma^2}}}{1-\frac{2\mu}{\sigma^2}},\quad& \mu\neq \frac{1}{2}\sigma^2\\
\log(x),\quad & \mu=\frac{1}{2}\sigma^2.
\end{cases}
\]

We have two trivial cases:
\begin{enumerate}[1.]
	\item \emph{Assume that $\mu>\frac{1}{2}\sigma^2$.} Then $S(\infty)=0$ and so $\infty$ is attractive (by Proposition \ref{prop trans1}) and $V(x)=\infty$ by Corollary \ref{cor inf}.
	\item \emph{Assume that $\mu= \frac{1}{2}\sigma^2$.} Then $-S(0)=\infty=S(\infty)$, and gBm is recurrent leading to a value $V(x)=\infty$ (by Lemma \ref{lemma recurrent}).
	\end{enumerate}
	
	The third case is the most interesting one:
	\begin{enumerate}[1.]\setcounter{enumi}{2}
		\item \emph{Assume that $\mu<\frac{1}{2}\sigma^2$.} Then $S(0)=0$ and $S(\infty)=\infty$, so that $0$ is attractive while $\infty$ is not. Furthermore $\p_x\left(\tau_b<\infty\right)=\frac{S(x)}{S(b)}$ so that 
	\begin{align*}
	\lim_{b\to\infty}\p_x\left(\tau_b<\infty\right)b^2=\lim_{b\to\infty}\frac{b^2}{S(b)}S(x)=\begin{cases}
		\infty,\quad & \mu>-\frac{1}{2}\sigma^2\\
	x^2\in(0,\infty),\quad & \mu=-\frac{1}{2}\sigma^2\\
		0,\quad & \mu<-\frac{1}{2}\sigma^2.
	\end{cases}
	\end{align*}
	Hence we have yet another three cases with the first one being trivial:
	\begin{enumerate}[$i.)$]
		\item \emph{Assume further that $\mu>-\frac{1}{2}\sigma^2$.} Then $\lim_{b\to\infty}\p_x\left(\tau_b<\infty\right)b^2=\infty$ and consequently $V(x)=\infty$ by Proposition \ref{prop inf}.
		\item \emph{Assume further that $\mu<-\frac{1}{2}\sigma^2$.} Then $\lim_{b\to\infty}\p_x\left(\tau_b<\infty\right)b^2=0$ and consequently all the conditions of Case (I) in Assumption \ref{assumption} are satisfied. In addition, as $\frac{S'(z)}{S(z)}z\equiv 1-\frac{2\mu}{\sigma^2}$ is a constant and hence non-decreasing, we apply Theorem \ref{theo main}\eqref{theo main Sim}: For a fixed $x>0$, the optimal stopping time is $\tau_{(0,z^*(x))}$, where $z^*(x)$ is a unique solution to a first order optimality condition
		\begin{align}\label{eq gbm opt}
		\frac{S(z^*)-S(x)}{\frac{1}{2}S(z^*)-S(x)}=\frac{S'(z^*)}{S(z^*)}z^*\quad\Longleftrightarrow \quad z^*(x)=\left(\frac{2\mu}{\mu+\frac{1}{2}\sigma^2}\right)^{\frac{\sigma^2}{\sigma^2-2\mu}}x.
		\end{align}
		Moreover, the value reads as
		\begin{align*}
		V(x)&=z^*(x)^2\frac{S(x)}{S(z^*(x))}\left(1-\frac{S(x)}{S(z^*(x))}\right)\\
		&=x^2\left(\left(\frac{2\mu}{\mu+\frac{1}{2}\sigma^2}\right)^{\frac{\sigma^2+2\mu}{\sigma^2-2\mu}}-\left(\frac{2\mu}{\mu+\frac{1}{2}\sigma^2}\right)^{\frac{4\mu}{\sigma^2-2\mu}}\right).
		\end{align*}
		We notice that the optimal stopping time and the value are identical to what was obtained in Theorem 3.2 in \cite{Pedersen11}. Notice also that $z^*(x)$ from \eqref{eq gbm opt} never intersects the diagonal $x$ (as $2\mu \neq \mu +\frac{1}{2}\sigma^2$ always), so that there is no interesting dynamic optimal solution (introduced in Subsection \ref{subsec on optimality}) for the problem.
		\item \emph{Assume further that $\mu=-\frac{1}{2}\sigma^2$}, so that $S(x)=\frac{1}{2}x^2$ and $\lim_{b\to\infty}\p_x\left(\tau_b<\infty\right)b^2=x^2\in(0,\infty)$. Now we can apply Lemma \ref{lemma special}\eqref{lemma hard A} to conclude the result in this case. We see that the first order optimality condition 
		\[\frac{S(z)-S(x)}{\frac{1}{2}S(z)-S(x)}=\frac{S'(z)}{S(z)}z\quad\Longleftrightarrow \quad x^2=0\]
		does not have a solution for any $x\in(0,\infty)$. Consequently, by Lemma \ref{lemma special}\eqref{lemma hard A}, the optimal "stopping time" is $\tau_{(0,\infty)}$, and the value reads as
		\begin{align*}
		V(x)=\lim_{z\to\infty}z^2\frac{x^2}{z^2}\left(1-\frac{x^2}{z^2}\right)=x^2.
		\end{align*}
		Especially we see that this value is finite, but at the same time it is not attainable almost surely. However, by choosing $Z>x$ to be a large number, we get with a stopping time $\tau_{(0,Z)}$
		\begin{align*}
		\Var_x\left\{X_{\tau_{(0,Z)}}\right\}=x^2-\frac{x^4}{Z^4},
		\end{align*}
		which we can get as close to $V(x)$ as we like. Here $\tau_{(0,Z)}$ is a finite stopping time with probability $\p_x\left(\tau_Z<\infty\right)=\frac{x^2}{Z^2}$.
		
		We would like to mention that this case ($\mu=-\frac{1}{2}\sigma^2$) was not considered in \cite{Pedersen11}. 
	\end{enumerate}
\end{enumerate}

\subsection{Jacobi diffusion} Next we illustrate our results on a finite state space when both boundaries are attractive. To that end, let $\I=(0,1)$ and consider a Jacobi diffusion $X_t$ (see e.g. Chapter 2 in \cite{Kuznetsov04} for a basic characteristics), which is a solution to a SDE
\begin{align*}
dX_t=(a-bX_t)dt+\sigma\sqrt{X_t(1-X_t)}dW_t.
\end{align*} 
Here $W_t$ is a standard Brownian motion. Moreover, we assume that $a,b,\sigma\in\R_+$ are such that $0<\frac{a}{b}<1$, so that the mean-reverting level lies in the interval $(0,1)$. Furthermore, for illustrative purposes, we assume that $2b-2a<\sigma^2$ and $2a<\sigma^2$ so that we can write down the scale function explicitly as
\begin{align*}
S(x)=\text{Beta}(x,-B,-A),
\end{align*}
where $\text{Beta}$ is the incomplete beta function, $B:=\frac{2a}{\sigma^2}-1\in(-1,0)$, and $A:=\frac{2b}{\sigma^2}-\frac{2a}{\sigma^2}-1\in(-1,0)$. Now $S(0)=0$ and $S(1)<\infty$ so that both end points are attractive and we have Case (III) of Assumption \ref{assumption} to consider and the solution can be read from Theorem \ref{theo mainIII}. Especially, as the state space is finite, the value of a variance stopping problem \eqref{eq prob} is always finite.

Notice that in \cite{Pedersen11} the Jacobi diffusion was also examined but in a case where only the lower boundary $0$ was an attractive point instead of both end points. This affects greatly to the outcome as in our case, following Theorem \ref{theo mainIII}, the solution depends on which boundary is closer, and the closeness is measured by inspecting whether $\E_x\left\{X_{\tau_{(0,1)}}\right\}=\frac{S(x)}{S(1)}$ is greater or smaller than $\frac{1}{2}$. It can be proved that the monotonicities of $\frac{S'(z)z}{S(z)}$ and $\frac{S'(y)}{S(1)-S(y)}(1-y)$ are satisfied so that the solution is of the type:
\begin{align*}
\begin{aligned}
&\tau_{(0,z)},\quad & \text{if }S(x)\leq \tfrac{1}{2}S(1);\\
&\tau_{(y,1)},\quad & \text{if }S(x)> \tfrac{1}{2}S(1).
\end{aligned}
\end{align*}
Notice that in \cite{Pedersen11}, where only $0$ was an attractive point, the optimal stopping time was always of the type $\tau_{(0,z)}$. 

To illustrate numerically this example on Jacobi diffusion, let us choose $a=0.02$, $b=0.038$ and $\sigma=0.26$. Then the the mean-reverting level $\frac{a}{b}\approx 0.53\in(0,1)$. With these choices $A\approx-0.47$ and $B\approx-0.41$, and the state $S^{-1}(\frac{S(1)}{2})\approx 0.43$. Below this state, the optimal stopping time is $\tau_{(0,z^*(x))}$ and above it is $\tau_{(y^*(x),1)}$. The optimal stopping boundaries $z^*(x)$ and $y^*(x)$ are illustrated in Figure \ref{fig ex jacobi}.
\begin{figure}[!ht]
\begin{center}
\includegraphics[width=0.45\textwidth]{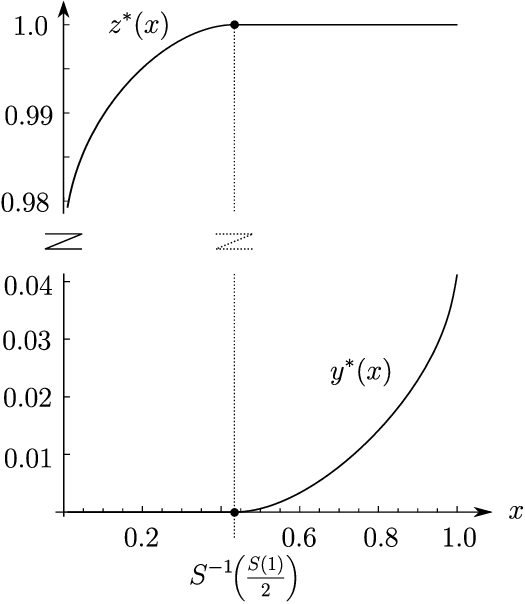}
\end{center}
\caption{\small \emph{Jacobi diffusion} -example  with  $a=0.02$, $b=0.038$ and $\sigma=0.26$. The optimal stopping time is $\tau^*=\tau_{(z^*,y^*)}$, if we interpret $z^*(x)=1$ for all $x\geq S^{-1}(\frac{S(1)}{2})$ and $y^*(x)=0$ for all $x\leq S^{-1}(\frac{S(1)}{2})$.}\label{fig ex jacobi}
\end{figure}

\subsection{Randomized solution}\label{subsec ex random}
The mapping $\frac{S'(x)}{S(x)}x$ is non-decreasing with most of the usual diffusions, and consequently the solution is a "pure strategy" stopping time with the most familiar diffusions. 

In order to illustrate the randomized stopping time -concept, we construct a specific diffusion: Let a state space be $\I=\R_+$ and define the scale function by
\begin{align*}
S(x):=\begin{cases}
\frac{x^2-\frac{3}{2}x}{4x-6},\quad&x<2\\
\frac{x^2-\frac{3}{2}x}{-10x+22},\quad&x\in[2,2.1)\\
\frac{x^2-\frac{3}{2}x}{\frac{1}{10}x+0.8},\quad&x\in[2.1,12)\\
\frac{x^2-\frac{3}{2}x}{2e^{12}e^{-x}},\quad&x\geq 12.
\end{cases}
\end{align*}
One can easily check that such an $S(x)$ is increasing, continuous, $S(0)=0$, and $S(\infty)=\infty$, so that $0$ is attractive and $\infty$ is not. Moreover, we can straightforwardly check that $\lim_{b\to\infty}\frac{b^2}{S(b)}=0$ and conclude that the conditions of Case (I) in Assumption \ref{assumption} are satisfied. Notice that $S$ is not continuously differentiable over the points $2,$ $2.1,$ and $12$, but the general proof does not require differentiability so that we can now apply Theorem \ref{theo main}\eqref{theo main Gen}. Observe that also the monotonicity condition of $\frac{S'(x)}{S(x)}x$ is not met, as it is strictly decreasing on $(2.1,12)$.

We now follow the algorithm from Subsection \ref{subsec algo}.
\begin{enumerate}[Step 1.]
	\item \begin{enumerate}[(i)]
		\item For $c>0$, we solve the auxiliary embedded quadratic problem 
\begin{align*}
V^c(x)=\sup_\tau\E_x\left\{(X_\tau-c)^2\right\}.
\end{align*}
 It can be shown that in this particular example, for a given $c\in\R_+\setminus\{\bar{c}\}$, $\bar{c}= \frac{3}{4}$, there exists a unique state $z_c$ maximizing a ratio $\frac{z^2-2cz}{S(z)}$ (cf. Lemma \ref{lemma stop}). Moreover, $z_c\equiv 2$ for all $c<\bar{c}$ and $z_c\equiv 12$ for all $c\in(\bar{c},5.54)$. Now $C=\{\bar{c}\}$. 
\item For a $\bar{c}$ there exist two states $2= \underline{z}<\overline{z}=12$ both maximizing the ratio $\frac{z^2-2\bar{c}z}{S(z)}$. Now $\E_2\left\{X_{\tau_{(0,12)}}\right\}\approx 0.095<0.75=\bar{c}$ so that Assumption \ref{assumption 2}\eqref{as2 Case I} does not hold. Thus we need to continue to Step 2b.
	\end{enumerate}
\item[Step 2b.] \begin{enumerate}[(i)]
		\item Now we can solve $\underline{x}$ and $\overline{x}$:
		\begin{align*}
		\E_{\underline{x}}\left\{X_{\tau_{(0,\underline{z})}}\right\}=\frac{S(\underline{x})}{S(2)}2={\bar{c}} &\quad \Longrightarrow\quad \underline{x}=S^{-1}(0.1875)=0.75\\
		\E_{\overline{x}}\left\{X_{\tau_{(0,\overline{z})}}\right\}=\frac{S(\overline{x})}{S(12)}12={\bar{c}} &\quad \Longrightarrow\quad \overline{x}=S^{-1}(3.9375)\approx 2.958.
\end{align*}
\item For $x\in(\underline{x},\underline{z}]=(0.75,2]$, we can use a randomized stopping time $\hat{\tau}(p)=\xi(p)\tau_{(0,\underline{z})}+(1-\xi(p))\tau_{(0,\overline{z})}$ to produce an optimal stopping time. However, for $x\in(\underline{z},\overline{x})\approx(0.75,2.958)$, we need a randomized stopping time $\check{\tau}(p)=\xi(p)\tau_{D_{\bar{c}}}+(1-\xi(p))\tau_{(0,\overline{z})}$ so that we need to solve $D_{\bar{c}}$ for the embedded problem with a parameter $\bar{c}$. 

Applying standard optimal stopping arguments (e.g. from \cite{Salminen85} or \cite{PesShi06}) we can conclude that for $x\in(\underline{z},\overline{x})$, $\tau_{D_{\bar{c}}}=\tau_{(a^*,b^*)}=\tau_{(2,12)}$ is an optimal stopping time for $V^{\bar{c}}(x)$.
\item[(iv)] Now $\mathcal{J}_{\not{A}}=(\underline{x},\overline{x})\approx(0.75,2.958)$.
\end{enumerate} 
	
\item An optimal stopping time is
\begin{align*}
\begin{cases}
\tau_{(0,z_*(x))},\quad & x\in(0,\underline{x}]\quad  \left((0,0.75]\right)\\
\xi_{p^*_x}\tau_{(0,\underline{z})}+(1-\xi_{p^*_x})\tau_{(0,\overline{z})},\quad & x\in (\underline{x},\underline{z})\quad\left(=(0.75,2]\right)\\
\xi_{p^*_x}\tau_{(2,12)}+(1-\xi_{p^*_x})\tau_{(0,\overline{z})},\quad & x\in (\underline{z},\overline{x})\quad\left(\approx(2,2.958)\right)\\
\tau_{(0,z^*(x))},\quad & x\in[\overline{x},\infty)\quad \left(\approx[2.958,\infty)\right).
\end{cases}
\end{align*}
Here $z_*(x)$ is the smallest and $z^*(x)$ the greatest solution to the first order optimality condition
\[\frac{S(z)-S(x)}{\frac{1}{2}S(z)-S(x)}=\frac{S'(z)}{S(z)}z.\]
Moreover, for $x\in(0.75,2]$, $p^*_x$ is the unique solution to
\begin{align*}
\E_x\left\{X_{\hat{\tau}(p)}\right\}=\bar{c} \Longleftrightarrow p^*_x\approx \frac{0.7875}{x}-0.05.
\end{align*} 
For $x\in(2,2.958)$, $p^*_x$ is the unique solution to 
\begin{align*}
\E_x\left\{X_{\check{\tau}(p)}\right\}=\bar{c} \Longleftrightarrow p^*_x\approx 6.25 -\frac{369.14}{63-S(x)}.
\end{align*}

\label{ex step 3}
\item The value reads as
\begin{align*}
V(x)=\begin{cases}
z_*(x)^2\frac{S(x)}{S(z_*(x))}\left(1-\frac{S(x)}{S(z_*(x))}\right),\quad &x\in(0,\underline{x}]\\
0.5625+2S(x),\quad&x\in(\underline{x},\overline{x})\\
z^*(x)^2\frac{S(x)}{S(z^*(x))}\left(1-\frac{S(x)}{S(z^*(x))}\right),\quad &x\in[\overline{x},\infty),
\end{cases}
\end{align*}
where $z_*(x)$ and $z^*(x)$ are as in Step \ref{ex step 3} above. 
\end{enumerate}

In this example we saw how an optimal stopping time can be a mixture between two different types of stopping times, namely $\tau_{z}$ and $\tau_{(a,b)}$. There are also examples where we randomize between $\tau_z$ and $0$ (see e.g. \cite{GadPed15}). In this way we see how the variance stopping problem can offer surprising solutions despite its simple formulation.

\subsection*{Acknowledgements}
The authors gratefully acknowledge the many helpful suggestions by the two anonymous referees. They clearly improved the quality of the paper. Also, the discussions with many colleagues are greatly appreciated.

\bibliographystyle{amsplain}
\bibliography{VarianceBib}

\end{document}